\documentclass[12pt]{amsart}
\usepackage[all]{xy}
\usepackage{graphicx}
\usepackage{tikz}  

\usepackage{amsmath}
\usepackage{amssymb}
\usepackage{amsfonts}
\usepackage{mathrsfs}
\usepackage{mathtools}

\usepackage{amsmath,amssymb,amsfonts,amsthm,stackengine,graphicx,caption}	
\usepackage{color}	

\def \N {\mathbb N}

\def \C {\mathbb C}
\def \D {\mathbb D}
\def \T {\mathbb T}
\def \P {\mathcal P}

\newcommand\abs[1]{\left|#1\right|}

\newcommand{\norm}[1]{\Vert#1\Vert}
\newcommand{\bignorm}[1]{\bigl\Vert#1\bigr\Vert}
\newcommand{\Bignorm}[1]{\Bigl\Vert#1\Bigr\Vert}
\newcommand{\biggnorm}[1]{\biggl\Vert#1\biggl\Vert}

\newcommand{\HI}{H^\infty}

\theoremstyle{plain}
\newtheorem{theorem}{Theorem}[section]
\newtheorem{proposition}[theorem]{Proposition}
\newtheorem{corollary}[theorem]{Corollary}
\newtheorem{definition}[theorem]{Definition}
\newtheorem{lemma}[theorem]{Lemma}
\theoremstyle{definition}
\newtheorem{remark}[theorem]{Remark}

\begin{document}

\title[Operators with finite peripheral spectrum]
{Polygonal functional calculus for operators with finite peripheral spectrum}

\author[O. Bouabdillah]{Oualid Bouabdillah}
\email{oualid.bouabdillah@univ-fcomte.fr,}
\address{Laboratoire de Math\'ematiques de Besan\c con, UMR 6623, 
CNRS, Universit\'e Bourgogne Franche-Comt\'e,
25030 Besan\c{c}on Cedex, FRANCE}

\author[C. Le Merdy]{Christian Le Merdy}
\email{clemerdy@univ-fcomte.fr}
\address{Laboratoire de Math\'ematiques de Besan\c con, UMR 6623, 
CNRS, Universit\'e Bourgogne Franche-Comt\'e,
25030 Besan\c{c}on Cedex, FRANCE}

\date{\today}

\maketitle

\begin{abstract}
Let $T\colon X\to X$ be a bounded operator on Banach 
space, whose spectrum
$\sigma(T)$ is included in the closed unit disc $\overline{\D}$.
Assume that the peripheral spectrum $\sigma(T)\cap\T$ is finite
and that $T$ satisfies a
resolvent estimate
$$
\norm{(z-T)^{-1}}\lesssim \max\bigl\{
\vert z -\xi\vert^{-1}\, :\,\xi\in \sigma(T)\cap\T\bigr\},
\qquad z\in\overline{\D}^c.
$$
We prove that $T$ admits a bounded polygonal
functional calculus, that is, an estimate
$\norm{\phi(T)}\lesssim \sup\{\vert\phi(z)\vert\, :\, z\in\Delta\}$ for 
some polygon $\Delta\subset\D$ and all polynomials $\phi$, in 
each of the following two cases : 
(i) either $X=L^p$ for some 
$1<p<\infty$, and $T\colon
L^p\to L^p$ is a positive contraction;
(ii) or $T$ is polynomially bounded
and for all $\xi\in \sigma(T)\cap\T,$ there 
exists a neighborhood $\mathcal V$ of
$\xi$ such that the set
$\{(\xi-z)(z-T)^{-1}\, :\,
z\in{\mathcal V}\cap \overline{\D}^c\}$
is $R$-bounded (here $X$ is arbitrary).
Each of these two results extends 
a theorem of de Laubenfels concerning
polygonal
functional calculus on Hilbert space. 
Our investigations require the introduction,
for any finite set $E\subset\T$, of a  notion of
Ritt$_E$ operator which generalises the classical
notion of Ritt operator. We study these Ritt$_E$ operators
and their natural functional calculus.
\end{abstract}




\bigskip
\section{Introduction} Let $X$ be a Banach space, 
let $T\colon X\to X$ be a bounded 
operator and let $\Omega\subset\C$ be an 
open set whose closure contains $\sigma(T)$, the spectrum of $T$.
In various situations, an important issue is to determine
whether there exists a constant 
$K\geq 1$ such that:
\begin{equation}\label{VN}
\text{For all polynomial}\ \phi,\quad 
\norm{\phi(T)}\leq K\sup\bigl\{\vert\phi(z)\vert\, :\, z\in\Omega\bigr\}.
\end{equation}
The search for such functional calculus estimates 
stemmed from the famous von Neumann inequality which asserts
that if $X=H$ is a Hilbert space and $\norm{T}\leq 1$, then
(\ref{VN}) holds true with $\Omega=\D$, the unit disc of $\C$,
and $K=1$. In Hilbertian operator theory, several important topics 
are related to von Neumann's inequality and 
to the search for inequalities of the form (\ref{VN}). This 
includes the study of polynomial boundedness, $K$-spectral
sets and similarity problems, for which we refer to 
\cite{Ba1,BBC,CG,DEY,Pau,Pis} and the references therein.

We recall that $T\colon X\to X$ is called polynomially bounded if
there exists a constant 
$K\geq 1$ such that (\ref{VN}) holds true with $\Omega=\D$.
In this paper we are interested in the case when 
the open set $\Omega\subset\C$ in (\ref{VN}) is a polygon. 
More explicitly, we say that $T\colon X\to X$ admits a bounded polygonal 
functional calculus if there exist a (convex, open) polygon
$\Delta\subset\D$ such that 
$\sigma(T)\subset \overline{\Delta}$, 
and a constant $K\geq 1$ such that 
(\ref{VN}) holds true with $\Omega=\Delta$.

Let $\T=\{z\in\C\, :\, \vert z\vert=1\}$. For any
$T$ such that $\sigma(T)\subset\overline{\D}$,
call 
$\sigma(T)\cap\T$ the peripheral spectrum of $T$.
It is easy to check (see Remark \ref{PolygonToRitt})
that if $T\colon X\to X$ admits a bounded polygonal 
functional calculus, then $\sigma(T)\cap\T$ is finite and one has
an estimate
\begin{equation}\label{ResEst}
\norm{(z-T)^{-1}}\lesssim \max\bigl\{
\vert z -\xi\vert^{-1}\, :\,\xi\in \sigma(T)\cap\T\bigr\},
\qquad z\in\overline{\D}^c.
\end{equation}
In the Hilbertian case,
the following converse was established by de Laubenfels 
\cite[Theorem 4.4, $(a)\Rightarrow (b)$]{dL}.

\begin{theorem}\label{TH0}
Let $H$ be a Hilbert space and let $T\colon H\to H$ be a
polynomially bounded operator. Assume that $T$ has a finite
peripheral spectrum and that (\ref{ResEst}) holds true.
Then $T$ admits a bounded polygonal 
functional calculus.
\end{theorem}

This remarkable result also follows from
\cite[Theorem 5.5]{FMI}. Note that it is already significant when 
$T\colon H\to H$ is a contraction.

The motivation for this paper is to understand polygonal 
functional calculus in the Banach space setting.

Our first main result is a Banach 
space version of Theorem \ref{TH0} relying on the notion of 
$R$-boundedness (for which we refer e.g. to \cite[Chapter 8]{HVVW}).
We prove (see Corollary
\ref{Main1} and Remark \ref{RRE}) that if a polynomially bounded
operator $T\colon X\to X$ has a
finite peripheral spectrum and if for any
$\xi\in\sigma(T)\cap\T$, there exists a 
neighborhood ${\mathcal V}$ of $\xi$ such that 
the set 
$$
\bigl\{(\xi-z)(z-T)^{-1}\, :\,
z\in{\mathcal V}\cap \overline{\D}^c\bigr\}
$$
is $R$-bounded, then 
$T$ admits a bounded polygonal functional 
calculus. We will see that such a  bounded functional 
calculus holds with respect to a polygon $\Delta\subset\D$
such that $\overline{\Delta}\cap\T=\sigma(T)\cap\T$ 
(see Remark \ref{Sharp}). In the case when the
peripheral spectrum $\sigma(T)\cap\T$ is a singleton,
this result reduces to \cite[Proposition 7.7]{LM1}, established for Ritt operators.

Our second main result 
concerns positive contractions 
on $L^p$-spaces. Let $1<p<\infty$, let $(S,\mu)$ be a measure space
and let $T\colon L^p(S)\to L^p(S)$ be 
a positive contraction. We show (see Theorem \ref{MainLP})
that if $T$ has a
finite peripheral spectrum and if (\ref{ResEst}) holds true,
then
$T$ admits a bounded polygonal 
functional calculus. Note that in this result,
we do not need to assume that $T$ is polynomially bounded.
The latter comes as a consequence of the bounded polygonal 
functional calculus. This result, which holds as well
if $T\colon L^p(S)\to L^p(S)$ is a contractively regular operator,
should be regarded as an $L^p$-version of 
the case when $\norm{T}\leq 1$ in Theorem \ref{TH0}.

In order to study operators with finite peripheral spectrum and
to obtain the aforementioned results, it is relevant to introduce,
for a finite set $E=\{\xi_1,\ldots,\xi_N\}\subset\T$, the notion of
Ritt$_E$ operator, see Definition \ref{DefRittE} below. This
is a natural generalization of Ritt operators, the latter having attracted a
lot of attention recently, see e.g. \cite{AFL,ALM,Bl,GT,KP,LM1,
LMX,Ly,NZ,N,V}. Section 2 is 
devoted to the study of
Ritt$_E$ operators. 
In particular we establish an analogue of 
the well-known theorem which asserts that $T\colon X\to X$ is a Ritt 
operator if and only if the two sets  
$\{T^n\, :\, n\geq 0\}$ and $\{n(T^n-T^{n-1})\, :\, n\geq 1\}$ are bounded.
This is achieved in Theorem \ref{THcaracterization}. In Sections
3 and 4, we relate the polygonal functional calculus to a natural
functional calculus associated with Ritt$_E$ operators, and prove
our main results.


\smallskip
\section{Ritt$_E$ operators}
Let $X$ be a Banach space. We let $B(X)$ denote the Banach algebra of all bounded 
operators on $X$, equipped with its usual norm. We denote by
$I_X$ the identity operator on $X$. For any 
$T\in B(X)$, we let $\sigma(T)$ denote the spectrum of $T$ and we let 
$R(z,T)=(z-T)^{-1}$ denote the resolvent operator, 
when $z$ belongs to the resolvent set $\C\setminus \sigma(T)$.

For any $a\in\C$ and any $r>0$, we let $D(a,r)\subset\C$ denote 
the open disc of radius $r$ centered at $a$. 
We recall that we use the notations ${\mathbb D}=D(0,1)$ and $\T$ for the 
the open unit disc and for its boundary, respectively.

\smallskip
\subsection{Definition and basic facts}
We consider distinct complex numbers
$\xi_1 , ... , \xi_N$ of modulus 1, for some $N\geq 1$, and we let 
\begin{equation}\label{E}
E = \{ \xi_1 , ... , \xi_N \}\subset\T.
\end{equation}

\begin{definition}\label{DefRittE}
Let $X$ be a Banach space and let $T \in B(X)$. 
We say that $T$ is Ritt$_E$ (or is a Ritt$_E$ operator) if 
$\sigma(T)\subset \overline{\D}$ and there exists a constant $c>0$ such that 
\begin{equation}\label{RE}
\|R(z,T)\| \leq\, \frac{c}{\prod\limits_{j=1}^N \abs{\xi_j - z}},
\qquad
z\in\C,\ 1<\vert z\vert <2.
\end{equation}
\end{definition}

Note that when $N=1$ and $E=\{1\}$, Ritt$_E$ 
operators coincide with the usual Ritt operators, for which
we refer e.g. to \cite{AFL,ALM,Bl,GT,KP,LM1,
LMX,Ly,NZ,N,V}.

\begin{lemma}\label{PRspectrum}
If $T\in B(X)$ is a Ritt$_E$ operator, then 
$\sigma(T) \subset \D \cup E$.
\end{lemma}

\begin{proof}
Let $z\in \sigma(T) \cap \T$.
For any $\lambda\in\C$ with $\vert \lambda\vert >1$, we have
$$
\norm{R(\lambda,T)}\geq\,\frac{1}{{\rm dist}(\lambda,\sigma(T))}\,.
$$
Further for any integer $t>0$,
${\rm dist}\bigl((1+t)z,\sigma(T)\bigr)\leq t$. Hence 
assuming (\ref{RE}), we have
$$
\prod\limits_{j=1}^N \bigl\vert\xi_j - (1+t)
z\bigr\vert\,\leq\, ct,
\qquad t>0.
$$
Letting $t\to 0$, this implies that $z=\xi_j$ for some $j\in\{1,\ldots,N\}$.
\end{proof}

\begin{remark}\label{RKradius}
Let $T\in B(X)$ such that $\sigma(T)\subset \overline{\D}$.

\smallskip
(a) 
For any $R>1$, (\ref{RE}) is equivalent to the existence of some constant 
$c_R >0$ such that 
$$
\|R(z,T)\| \leq \,\frac{c_R}{\prod\limits_{j=1}^N \abs{\xi_j - z}},
\qquad
z\in\C,\ 1<\vert z\vert <R.
$$

\smallskip
(b) We have $\norm{R(z,T)} =O(\vert z\vert^{-1})$ when $\vert z\vert\to\infty$, hence 
(\ref{RE}) is equivalent to the existence of some 
constant $C>0$ such that 
$$
\|R(z,T)\| \leq
\,C\, \max\limits_{j\in\{1,..,N\}} \frac{1}{\abs{\xi_j-z}}\,,
\qquad z\in \overline{\D}^c.
$$

\smallskip
(c)  Using Lemma  \ref{PRspectrum},  we see that
the existence of a finite subset $E\subset\T$ such that 
$T$ is Ritt$_E$ is equivalent to the following:
\begin{itemize}
\item[-] The peripheral spectrum $\sigma(T)\cap\T$ is finite and there exists  $C>0$ such that 
$$
\|R(z,T)\| \leq
\,C\, \max\biggl\{\frac{1}{\abs{\xi-z}}\, : \, \xi\in \sigma(T)\cap\T\biggr\},
\qquad z\in \overline{\D}^c.
$$
\end{itemize}
\end{remark}


\smallskip
\subsection{Auxiliary tools}\label{Aux}
For any $\omega\in(0,\pi)$, we let
$$
\Sigma_\omega=\bigl\{\lambda\in\C^*\, :\, \vert{\rm Arg}(
\lambda)\vert <\omega\bigr\}
$$
be the open sector of angle $2\omega$ around the positive real axis.

Given any $A\in B(X)$, we say that $A$ is sectorial of type $\omega$ if 
$\sigma(A)\subset\overline{\Sigma_\omega}$ and
for any $\nu\in(\omega,\pi)$, there exists a constant $K_\nu>0$ such that
$$
\norm{\lambda R(\lambda,A)}\leq K_\nu,\
\qquad\lambda\in\overline{\Sigma_\nu}^c.
$$
It is well-known that if $A\in B(X)$ is such that 
$\sigma(A)\subset\overline{\Sigma_{\frac{\pi}{2}}}$ and
$\{\lambda R(\lambda,A)\,:\,\lambda\in
\overline{\Sigma_{\frac{\pi}{2}}}^c\}$ is bounded, then
there exists $\omega\in \bigl(0,\frac{\pi}{2}\bigr)$
such that $A$ is sectorial of type $\omega$.
We will use this property in the next proof.
We refer e.g. to \cite[Section 10]{HVVW} 
for information on sectorial operators.

\begin{lemma}\label{Sect}
Assume that $T\in B(X)$ is a Ritt$_E$ operator. Then
for any $j=1,\ldots,N$, the operator
\begin{equation}\label{Aj}
A_j = I_X -\overline{\xi_j} T
\end{equation}
is a sectorial operator of type $<\frac{\pi}{2}$.
\end{lemma}

\begin{proof}
Let 
$\lambda\in \overline{\Sigma_{\frac{\pi}{2}}}^c$. Then
$\xi_j(1-\lambda)\in\overline{\D}^c$ and 
an elementary computation shows
that $\lambda$ belongs to the resolvent set of $A_j$, with
\begin{equation}\label{Transfer1}
\lambda R(\lambda, A_j) =  -\lambda\xi_j R(\xi_j(1-\lambda),T).
\end{equation}
Hence assuming (\ref{RE}), we have
$$
\norm{\lambda R(\lambda,A_j)}\leq
\, \frac{c\,\vert\lambda\vert}{\prod\limits_{k=1}^N \abs{\xi_k - \xi_j(1-\lambda)}}
\, =\, \frac{c}{\displaystyle{\prod_{\substack{1 \leq k \leq N \\ k \neq j}}}
\abs{\xi_k - \xi_j(1-\lambda)}}\,.
$$
Now observe that 
for any $k\not= j$, the set 
$\{\vert\xi_k -\xi_j(1-\lambda)\vert\, :\,\lambda\in
\overline{\Sigma_{\frac{\pi}{2}}}^c\}$
is bounded away from $0$.
We deduce that
$\{\lambda R(\lambda, A_j)\, :\,\lambda\in
\overline{\Sigma_{\frac{\pi}{2}}}^c\}$ is bounded,
which implies that $A_j$ is sectorial of type $<\frac{\pi}{2}$.
\end{proof}

\begin{remark}\label{Transfer2}
For $T$ as above, (\ref{Transfer1}) is valid
for any $\lambda$ in the resolvent set of $A_j$. Equivalently,
for any $z$ in the resolvent set of $T$, we have
\begin{equation}\label{Transfer3}
(\xi_j-z) R(z,T)=\,
-\overline{\xi_j}(\xi_j-z)R\bigl(
\overline{\xi_j}(\xi_j-z), A_j\bigr).
\end{equation}
\end{remark}

\bigskip
When one studies Ritt operators and their functional calculus,
it is useful to consider, for any angle $\omega \in 
\bigl(0,\frac \pi 2\bigr)$, 
the Stolz domain $B_\omega\subset 1-\Sigma_\omega$ 
defined as  the interior of the convex hull of 1 
and the disc $D(0,\sin(\omega))$ (see e.g. \cite{ALM,LM1}).
We will introduce similar domains adapted to the study of Ritt$_E$ operators.

\begin{definition}\label{ER}
Let $E=\{\xi_1,\ldots,\xi_N\}$ as in (\ref{E}) and let $r \in (0,1)$. 
We let $E_r$ denote the interior of the convex hull of $D(0,r) \cup E$, 
see Figure \ref{fig:stolz}.

\begin{figure}[!h]
    \includegraphics[scale=0.25]{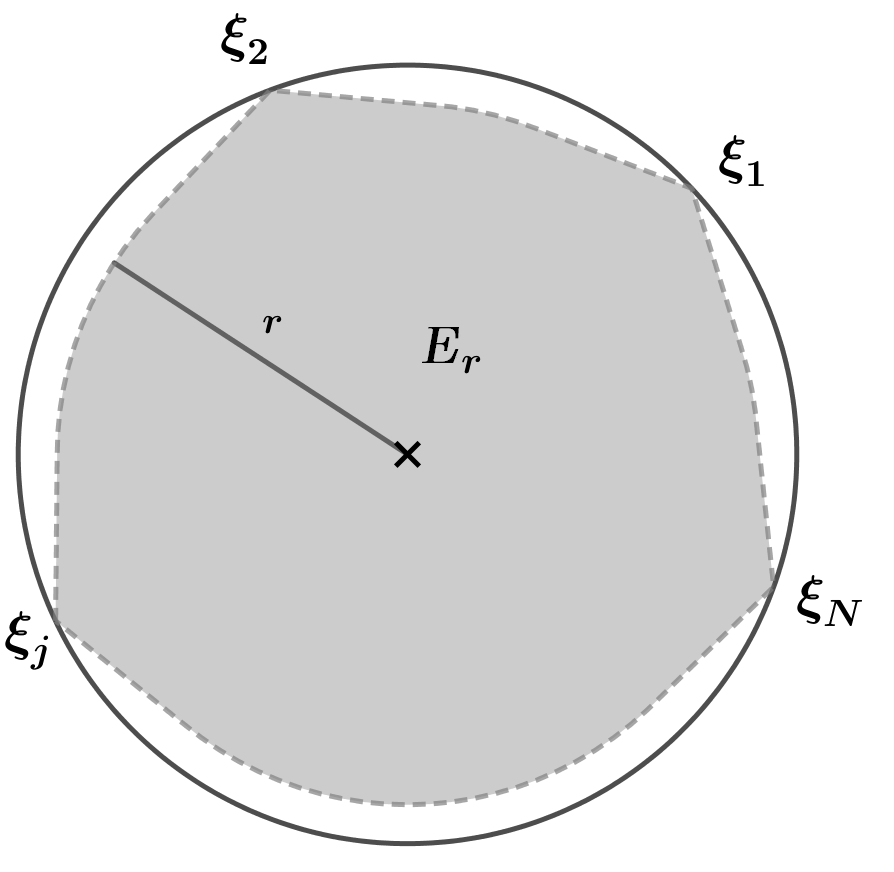}
    \caption{The ``generalized" Stolz domain $E_r$}
    \label{fig:stolz}
\end{figure}
\end{definition}

\begin{remark}\label{Warning}
Assume that $N\geq 2$ and that the sequence $(\xi_1,\xi_2,\ldots,\xi_N)$
is oriented counterclockwise on $\T$. Set $\xi_{N+1}=\xi_1$. We note that 
if $r\in(0,1)$ is such that $[\xi_j,\xi_{j+1}]\cap \overline{D}(0,r)\not=\emptyset$ 
for all $j=1,\ldots N$ (as shown on Figure \ref{fig:stolz}), then for all 
$r\leq u<s<1$, we have $\partial E_u\setminus E\subset E_s$.
In this case we will say that $r$ is $E$-large enough. 

If $N=1$, any 
$r\in(0,1)$ will be called $E$-large enough.
\end{remark}

For any $\xi\in\C^*$ and $\omega\in\bigl(0,\frac{\pi}{2}\bigr)$, we set
\begin{equation}\label{Sigma-xi}
\Sigma(\xi,\omega)  = \xi(1 -\Sigma_\omega).
\end{equation}
This is the open sector with vertex $\xi$ and angle
$2\omega$ around the semi-axis $[\xi,0)$, see Figure \ref{fig:polygonal sector}.

\begin{lemma}\label{LEstolz}
An operator $T \in B(X)$ is Ritt$_E$ if and only if there exists 
$r \in (0,1)$ which is $E$-large enough (in the sense of Remark \ref{Warning}),
such that the following two conditions hold:
\begin{itemize}
\item[(i)] $\displaystyle{\sigma(T) \subset \overline{E_r}}$;
\item[(ii)] For all $s\in(r,1)$, there exists a constant $c>0$
such that 
$$
\|R(z,T)\| \leq\, \frac{c}{\prod\limits_{j=1}^N \abs{\xi_j - z}},
\qquad
z\in D(0,2)\setminus \overline{E_s}.
$$
\end{itemize}
\end{lemma}

\begin{proof}
The `if part' is obvious. To prove the `only if' part,
let us assume that 
$T$ is a Ritt$_E$ operator.
For any  $j=1,\ldots,N$, consider 
$A_j$ given by (\ref{Aj}). According to Lemma \ref{Sect},
we may find $\omega\in\bigl(0,\frac{\pi}{2}\bigr)$
such that $A_j$ is sectorial of type $\omega$ for all 
$j=1,\ldots,N$. Then we have
$$
\sigma(T) \subset \overline{\Sigma(\xi_j,\omega)},\qquad j=1,\ldots N.
$$
Choosing $\omega$ close enough to $\frac{\pi}{2}$, we may  assume that we also have
$$
 E \subset \overline{\Sigma(\xi_j,\omega)},\qquad j=1,\ldots N.
$$

Choose $\eta>0$ small enough to ensure that 
$D(\xi_j,2\eta)\cap \Sigma(\xi_j,\omega)\subset\D$
for all $j=1,\ldots,N$.
Since $\sigma(T)$ is compact,
$$
F=\sigma(T) \setminus 
\left[ \bigcup\limits_{j=1}^N D(\xi_j,\eta) \right]
$$
is a compact subset of $\D$. Hence there exists $r\in(0,1)$ such that
$F\subset \overline{D}(0,r)$. We may and do assume that $r\geq\sin(\omega)$.
This ensures that $\sigma(T)\subset\overline{E_r}$ and that $r$ is $E$-large enough.

\begin{figure}[!h]
    \includegraphics[scale=0.3]{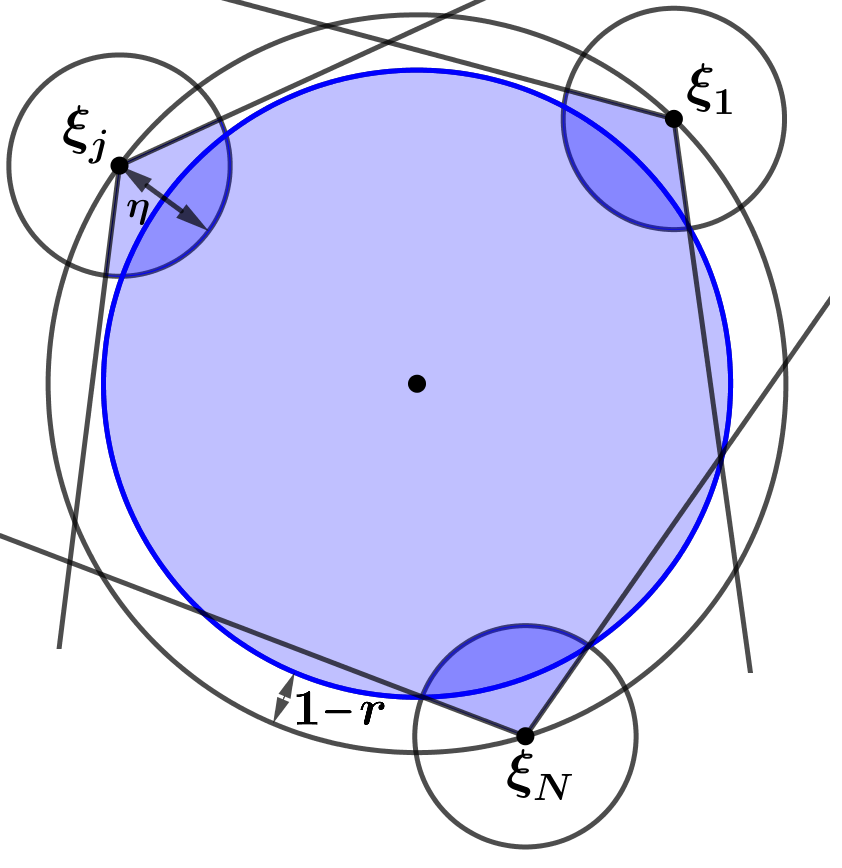}
    \caption{Set containing the spectrum of $T$}
    \label{fig:sectors}
\end{figure}

We set 
$$
h(z)=\,R(z,T)\,\prod\limits_{k=1}^N (\xi_k-z),
\qquad z\in\C\setminus\sigma(T).
$$
Let $s\in (r,1)$ and let $\beta = \arcsin(s)$. 
Then $\beta \in \bigl(\omega,\frac{\pi}{2}\bigr)$. For any $j$,
$A_j$ is sectorial of type $\omega$ hence 
by (\ref{Transfer3}),
$(\xi_j -z)R(z,T)$ is bounded on 
$\overline{\Sigma(\xi_j,\beta)}^c$. This implies that 
$h$ is bounded on $\D\cap \overline{\Sigma(\xi_j,\beta)}^c$. 
Thus if we let 
$$
F_0 = \bigcup\limits_{j=1}^N \overline{\Sigma(\xi_j,\beta)}^c,
$$
we obtain that $h$ is bounded on $\D\cap F_0$. Now consider
$$
F_1 = \biggl[\Bigl(\bigcap\limits_{j=1}^N
\overline{\Sigma(\xi_j,\beta)}\Bigr)\bigcap\bigl\{
s\leq\vert\lambda\vert\leq 1\bigr\}\biggr]\setminus  \bigcup\limits_{j=1}^N D(\xi_j,\cos(\beta)).
$$
By construction $F_1$
is a compact subset of $\C\setminus \sigma(T)$. Thus
$h$ is bounded on $F_1$.
Finally, we have 
$\D\setminus \overline{E_s}
\subset F_0\cup F_1$, hence $h$ is bounded on 
$\D\setminus \overline{E_s}$. Since $T$ is 
Ritt$_E$, this implies (ii).
\end{proof}

\begin{remark}
It follows from Lemmas \ref{PRspectrum} and
\ref{Sect} and from the proof of Lemma
\ref{LEstolz} that an operator $T\in B(X)$ is 
Ritt$_E$ if and only of the operators $A_1,\ldots,A_N$
defined
by (\ref{Aj}) are all sectorial of type $<\frac{\pi}{2}$,
and $\sigma(T)\subset \D\cup E$.
\end{remark}

\smallskip 
\subsection{A characterisation of Ritt$_E$ operators}
It is well-known that an operator $T\in B(X)$ is a Ritt operator if and only if 
the two sets $\{T^n\, :\, n\geq 0\}$ and $\{n(T^n-T^{n-1})\, :\, n\geq 1\}$ are bounded,
see \cite{Ly,NZ,N,V}. The next theorem is an extension of this result to Ritt$_E$ operators.

\begin{theorem}\label{THcaracterization}
An operator $T \in B(X)$ is Ritt$_E$ if and only if the following two conditions hold:
\begin{itemize}
\item[(i)] $T$ is power bounded, that is, there exists a constant $c_0\geq 1$ such that
$$
\|T^n\| \leq c_0, \qquad n \geq 0;
$$
\item[(ii)] There exists a constant $c_1>0$ such that 
$$
\biggnorm{T^{n-1}\prod\limits_{j = 1}^N(\xi_j-T)}\, 
\leq\,\frac {c_1}{n}\,,
\qquad n\geq 1.
$$
\end{itemize}
\end{theorem}

In order to prove the `if part', we will need the following
algebraic factorization property, which is a generalization
of a formula used in the
proof of \cite[Theorem 2.4]{Sei}.

\begin{lemma}\label{Seifert}
There exists a two variable complex polynomial 
$Q$ such that
for all $T\in B(X)$, all $\lambda \in \C^*$ and all $n \geq 1$,
$$
\lambda^n\prod\limits_{j=1}^N(\xi_j-\lambda) - T^n\prod\limits_{j=1}^N(\xi_j-T) = (\lambda-T) \prod\limits_{j=1}^N(\xi_j-\lambda)\lambda^{n-1} 
\sum\limits_{k=0}^{n-1}\lambda^{-k}T^k +T^n (\lambda-T)Q(\lambda,T).
$$
\end{lemma}

\begin{proof}
We have
\begin{align*}
\lambda^n\prod\limits_{j=1}^N(\xi_j-\lambda) - & T^n\prod\limits_{j=1}^N(\xi_j-T) \\ &= \lambda^n\prod\limits_{j=1}^N(\xi_j-\lambda) - T^n\prod\limits_{j=1}^N(\xi_j-\lambda) + T^n\prod\limits_{j=1}^N(\xi_j-\lambda) - 
T^n\prod\limits_{j=1}^N(\xi_j-T) \\
&= (\lambda^n-T^n)\prod\limits_{j=1}^N(\xi_j-\lambda) + T^n\biggl(\prod\limits_{j=1}^N(\xi_j-\lambda)-
\prod\limits_{j=1}^N(\xi_j-T)\biggr).
\end{align*}
Set 
$$
P(\lambda,z) = 
\prod\limits_{j=1}^N(\xi_j-\lambda)-
\prod\limits_{j=1}^N(\xi_j-z),
\qquad \lambda,z\in\C.
$$
Then we have $P(\lambda,z) = (\lambda-z) Q(\lambda,z)$ for 
some polynomial $Q$. Then 
$$
\prod\limits_{j=1}^N(\xi_j-\lambda)-\prod\limits_{j=1}^N(\xi_j-T) = (\lambda-T)Q(\lambda,T).
$$
We deduce that
\begin{align*}
\lambda^n\prod\limits_{j=1}^N(\xi_j-\lambda) - T^n\prod\limits_{j=1}^N(\xi_j-T) 
&= (\lambda^n-T^n)\prod\limits_{j=1}^N(\xi_j-\lambda) + T^n (\lambda-T)Q(\lambda,T)\\
&= (\lambda-T) \prod\limits_{j=1}^N(\xi_j-\lambda) \sum\limits_{k=0}^{n-1}\lambda^{n-1-k}T^k +T^n (\lambda-T)Q(\lambda,T),
\end{align*}
which yields the desired identity.
\end{proof}

\begin{proof}[Proof of Theorem \ref{THcaracterization}]

\

\underline{$\Rightarrow$ :} 
We assume that the operator $T$ is Ritt$_E$
and we will show (i) and (ii). The argument is an
adaptation of the one used for Ritt operators in \cite{V}.

Let $r\in (0,1)$ such that 
$T$ satisfies Lemma \ref{LEstolz}. Let $s\in (r,1)$, let $\alpha = \arcsin(s)$
and note that
$$
\vert 1 -\cos(\alpha)e^{\pm i\alpha}\vert =\sin(\alpha)=s.
$$
For convenience we assume that the sequence $(\xi_1,\xi_2,\ldots,\xi_N)$
is oriented counterclockwise on $\T$ and we set $\xi_{N+1}=\xi_1$.
Let $n\geq 1$. For any $j=1,\ldots,N$,
we define the following four paths (whose definitions
depend on $n$, although this is not reflected by the notation), see Figure 
\ref{fig:contour}:

\begin{itemize}
\item [-]
$\gamma_{j +}$ is the oriented segment $\bigl[\xi_j(1-\frac{\cos(\alpha)}{n}\,e^{-i\alpha}),
\xi_j(1- \cos(\alpha) e^{-i\alpha})\bigr]$;

\smallskip
\item [-]
$\gamma_{j -}$ is the oriented segment $\bigl[\xi_j(1-\cos(\alpha)e^{i\alpha}),
\xi_j(1- \frac{\cos(\alpha)}{n}\,e^{i\alpha})\bigr]$;

\smallskip
\item [-]
$\gamma_{j,j+1}$ is the simple path going from $\xi_j(1- \cos(\alpha) e^{-i\alpha})$
to $\xi_{j+1}(1-\cos(\alpha)e^{i\alpha})$ counterclockwise along the circle of centre $0$ and
radius $s$;

\smallskip
\item [-]
$\gamma_{j}$ is the simple path going from $\xi_j(1- \frac{\cos(\alpha)}{n}\,e^{i\alpha})$
to $\xi_j(1-\frac{\cos(\alpha)}{n}\,e^{-i\alpha})$ 
counterclockwise along the circle of centre $\xi_j$ and
radius $\frac{\cos(\alpha)}{n}$.
\end{itemize}

For $s$ close enough to $1$ and $n$ large enough, the concatenation of these $4N$ paths, that we denote by $\Gamma_n$,
is a Jordan 
contour enveloping $\sigma(T)$, and the support of $\Gamma_n$ is included in $E_s^c$.

\begin{figure}[!h]
    \includegraphics[scale=0.22]{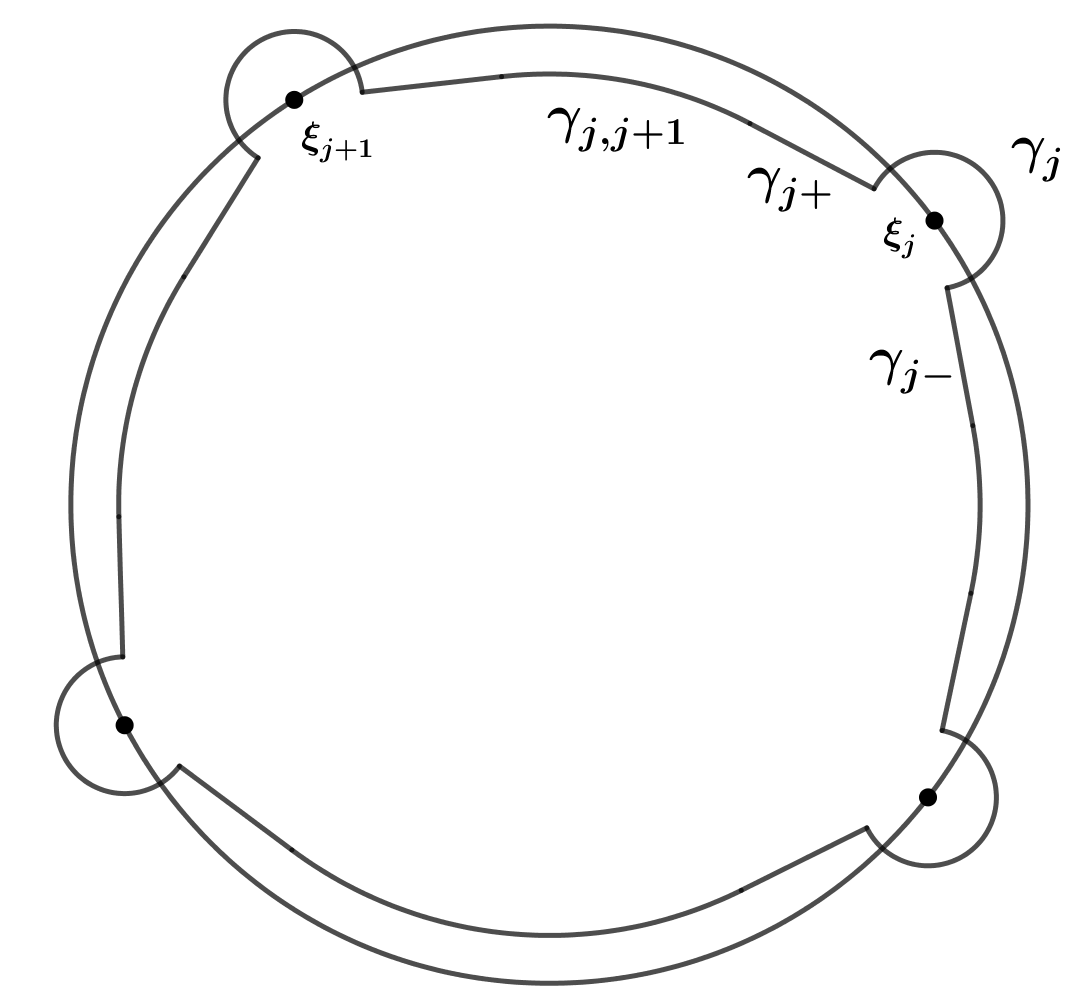}
    \caption{Integration contour}
    \label{fig:contour}
\end{figure}

We fix $s<1$ and $n_0\geq 1$ such that we are in the
above case for all $n\geq n_0$.
Then applying the Dunford-Riesz 
functional calculus, we have
$$
T^n = \,\frac{1}{2\pi i}\,
\int_{\Gamma_n} \lambda^n R(\lambda,T) \,d\lambda,
$$
for all $n\geq n_0$.
The upper bound in 
Lemma \ref{LEstolz}, (ii),
provides an estimate
$$
\norm{T^n} \leq\,\frac 1 {2\pi}\, 
\int_{\Gamma_n} \abs{ \lambda^n } \|R(\lambda,T)\| \,\vert d\lambda\vert
\,\leq \frac c {2\pi} \int_{\Gamma_n} \frac {\abs{ \lambda }^n}
{\prod\limits_{j = 1}^N \abs{ \lambda - \xi_j }} \,\vert d\lambda\vert\,.
$$
We will show that the integral in the 
right hand-side is uniformly bounded for $n\geq n_0$.

Fix $j_0 \in \{ 1,..,N\}$. For any $\lambda\in\gamma_{j_0,j_0+1}$,
we have $\vert\lambda\vert=s$ and hence $\vert\lambda-\xi_j\vert\geq 1-s$ for all
$j=1,\ldots,N$. 
Consequently,
$$
\int_{\gamma_{j_0,j_0+1}} 
\frac {\abs{ \lambda}^n} {\prod\limits_{j = 1}^N   
{\abs{ \lambda - \xi_j}}}\, \abs{d\lambda}\,
\leq \,\int_{\gamma_{j_0,j_0+1}} \frac {s^n} {(1-s)^N} \abs{d\lambda}
\,\longrightarrow\, 0,
$$
when $n\to\infty$.

Next observe that
$$
K=\sup\Biggl\{
\frac{1}{\displaystyle{\prod_{\substack{1 \leq j \leq N \\ j \neq j_0}} {\abs{\lambda - \xi_j}}}}\, 
:\, \vert\lambda-\xi_{j_0}\vert\leq\,\frac{\cos(\alpha)}{n}\Biggr\}\, <\infty.
$$
For all $ \lambda \in \gamma_{j_0}$, we have $\vert \lambda-\xi_{j_0}\vert =\,
\frac{\cos(\alpha)}{n}$ and hence $\vert\lambda\vert\leq 1+ \frac{1}{n}\,$. Consequently,	
\begin{align*}
\displaystyle {\int_{\gamma_{j_0}} 
\frac {\abs{\lambda}^n} {\prod\limits_{j = 1}^N  
{\abs{\lambda - \xi_j}}} \abs{d \lambda}}\, &=\,
\frac{1}{\cos(\alpha)}\,
\displaystyle {\int_{\gamma_{j_0}} \frac {n \abs{ \lambda}^n} 
{\displaystyle\prod_{\substack{1 \leq j \leq N \\ j \neq j_0}} {\abs{ \lambda - \xi_j}}} \abs{d \lambda}}  \\
&\leq \,\frac{K}{\cos(\alpha)}\,\displaystyle{ 
\int_{\gamma_{j_0}} {n \left( 1+ \frac 1 n \right)^n} \abs{d \lambda}} \\
& \leq \, 2\pi e K.
\end{align*}

Let us now focus on the integral over $\gamma_{j_0,+}$.
We may assume that $\xi_{j_0} = 1$ (otherwise, change
$T$ into  $\overline{\xi_{j_0}}T$.)
Observe that
$$
C=\sup\Biggl\{
\frac{1}{\displaystyle{\prod_{\substack{1 \leq j \leq N \\ j \neq j_0}} {\abs{\lambda - \xi_j}}}}\, 
:\, \lambda\in\bigl[1, 1 -\cos(\alpha)e^{-i\alpha}\bigr]
\Biggr\}\, <\infty.
$$
For all $t \in [0,\cos(\alpha)]$, 
we have $t^2 \leq t \cos(\alpha)$, hence
$$
\abs{ 1 - t e^{-i \alpha}}^2  = 1 + t^2 - 2t \cos(\alpha) \leq 1 - t\cos(\alpha).
$$
We derive the  estimate
$$
\int_{\gamma_{j_0,+}} \frac {\abs{\lambda}^n} {\prod\limits_{j = 1}^N {\abs{\lambda - \xi_j}}} \abs{d \lambda}
\leq C \displaystyle {\int_{\frac {\cos(\alpha)} n}^{\cos(\alpha)} \frac {(1-t\cos(\alpha))^{\frac n 2}} {t }\, d t }.
$$
Using the change of variable $t\to t {\cos(\alpha)}$ and the inequality $1-t\leq e^{-t}$, we have
\begin{align*}
\displaystyle{ \int_{\frac {\cos(\alpha)} n}^{\cos(\alpha)} \frac {(1-t\cos(\alpha))^{\frac n 2}} {t} \,d t }
&\leq  \displaystyle{ \int_{\frac {\cos^2(\alpha)} n}^{+\infty} \frac {e^{-{\frac {nt} 2}}} {t}\, d t}
\, =\, \displaystyle {\int_{\frac {\cos^2(\alpha)} 2}^{+\infty} \frac {e^{-t}} {t} \, d t}\,.
\end{align*}
This proves that the integrals 
$\displaystyle {\int_{\gamma_{j_0,+}} \frac {\abs{\lambda}^n} {\prod\limits_{j = 1}^N {\abs{\lambda - \xi_j}}} \abs{d \lambda}}$ are uniformly bounded for $n\geq n_0$. 
The same holds true if we replace $\gamma_{j_0,+}$ by $\gamma_{j_0,-}$.
Thus we have proved that 
$$
\sup_{n\geq n_0}\int_{\Gamma_n} \frac {\abs{ \lambda }^n} {\prod\limits_{j = 1}^N \abs{ \lambda - \xi_j }} \,\vert d\lambda\vert\,<\infty.
$$
This implies that  $\{T^n\, :\, n\geq 0\}$ is bounded, hence (i) is proved.

Let us now prove (ii), using $\Gamma_n$ as above. Applying
the Dunford-Riesz 
functional calculus again, we have
$$
\displaystyle{nT^{n-1} \prod\limits_{j = 1}^N (\xi_j - T) = 
\,\frac n {2\pi i}\,
\int_{\Gamma_n} \lambda^{n-1}\prod\limits_{j = 1}^N (\xi_j - \lambda) R(\lambda,T) \, d\lambda},
$$
for all $n\geq n_0$. This implies
$$
\Bignorm{nT^{n-1} \prod\limits_{j = 1}^N (\xi_j - T)}
\leq \frac {cn} {2\pi} 
\displaystyle{ \int_{\Gamma_n}\abs{ \lambda}^{n-1} 
\abs{\, d \lambda}}\,.
$$
Thus it suffices to show that the integrals in the right hand-side
are uniformly bounded for $n\geq n_0$.

As before we  fix $j_0\in\{1,\ldots,N\}$. 
For all $\lambda \in \gamma_{j_0,j_0+1}$,  
we have $\abs{ \lambda } = s$
hence
$$
n \int_{\gamma_{j_0,j_0+1}}\abs{ \lambda }^{n-1} 
\abs{d\lambda} \,\leq \,2\pi ns^{n-1}\,\longrightarrow\, 0,
$$
when $n\to\infty$.

Next for all $\lambda \in \gamma_{j_0}$, we have 
$\abs{ \lambda} \leq 1+\frac 1 n$ hence
$$
n  \int_{\gamma_{j_0}}\abs{ \lambda}^{n-1} \, \abs{d \lambda}\,
\leq n \Bigl(1 + \frac 1 n\Bigr)^{n-1} \bigl\vert \gamma_{j_0}\bigr\vert
\leq 2\pi e.
$$

Finally assume as before
that $\xi_{j_0}=1$. Then arguing as above we have
\begin{align*}
n \displaystyle{ \int_{\gamma_{j_0,+}}\abs{ \lambda}^{n-1} \abs{d \lambda}} 
&\leq n \displaystyle{ \int_{\frac {\cos(\alpha)} n}^{\cos(\alpha)} (1 - t\cos(\alpha))^{\frac {n-1} 2} d t}\\
&\leq \frac n {\cos(\alpha)} 
\int_{\frac {\cos^2(\alpha)} n}^{\infty} 
(1 - t)^{\frac {n-1} 2} \, d t  \\
&\leq  \frac n {\cos(\alpha)} \displaystyle{ 
\int_{\frac {\cos^2(\alpha)} n}^{\infty} {e^{\frac {-t(n-1)} 2}} dt}\\
&\leq \frac {4} {\cos(\alpha)} \displaystyle{ \int_{0}^{\infty} {e^{-t}} dt }
\,=\,\frac {4} {\cos(\alpha)}\,.
\end{align*}
We obtain that 
$n \displaystyle \int_{\gamma_{j_0,+}}\abs{ \lambda }^{n-1} 
\abs{d \lambda}$ has a uniform bound. 
The same holds true with $\gamma_{j_0,-}$ in
place of 
$\gamma_{j_0,+}$. These estimates imply (ii).

\smallskip
\underline{$\Leftarrow$ :} Assume (i) and (ii). The fact that 
$T$ is power bounded implies that $\sigma(T)\subset\overline{\D}$.
For any $\lambda \in \C$, with $1<\vert\lambda\vert<2$, we
multiply the identity of Lemma \ref{Seifert} by 
$R(\lambda,T)$ to obtain
$$
\lambda^n\prod\limits_{j=1}^N(\xi_j-\lambda) R(\lambda,T) =T^n\prod\limits_{j=1}^N(\xi_j-T) R(\lambda,T) + \prod\limits_{j=1}^N(\xi_j-\lambda) \lambda^{n-1} \sum\limits_{k=0}^{n-1}\lambda^{-k}T^k +T^n Q(\lambda,T).
$$
This implies
\begin{align*}
\abs{\lambda}^n\prod\limits_{j=1}^N \abs{\xi_j-\lambda} \|R(\lambda,T)\| &\leq
\Bignorm{T^n\prod\limits_{j=1}^N(\xi_j-T)}\norm{R(\lambda,T)}\\
& +\prod\limits_{j=1}^N \abs{\xi_j-\lambda} \abs{\lambda}^{n-1}
\sum\limits_{k=0}^{n-1}\abs{\lambda}^{-k}\norm{T^k}
+ \norm{Q(\lambda,T)}\norm{T^n}.
\end{align*}
By continuity of polynomials,
there exists a constant $c>0$ such that 
$$
\|Q(\lambda,T)\|\leq c,\qquad \lambda\in D(0,2).
$$
We derive the following estimate,
$$
\abs{\lambda}^n\prod\limits_{j=1}^N \abs{\xi_j-\lambda} \|R(\lambda,T)\| \leq \frac{c_1}{n} \|R(\lambda,T)\|+ c_0 \prod\limits_{j=1}^N \abs{\xi_j-\lambda} \abs{\lambda}^{n-1} \sum\limits_{k=0}^{n-1}\abs{\lambda}^{-k} + cc_0.
$$
Dividing both sides by 
$\abs{\lambda}^n$, and using the fact that
$\vert\lambda\vert>1$, we derive 
\begin{equation}\label{EQSIFERT}
\prod\limits_{j=1}^N \abs{\xi_j-\lambda} 
\|R(\lambda,T)\| \leq \frac{c_1}{n} \|R(\lambda,T)\|+ c_0\Bigl(n \prod\limits_{j=1}^N 
\abs{\xi_j-\lambda} + c\Bigr).
\end{equation}
Now for a given $\lambda\in D(0,2)\setminus \overline{\D}$,
we apply this estimate with $n-1$ equal to the integer
part of 
$\frac{2c_1}{\prod\limits_{j=1}^N\abs{\xi_j-\lambda}}$.
Thus 
$$
\frac{c_1}{n}\leq\,\frac12\, \prod\limits_{j=1}^N\abs{\xi_j-\lambda}
\qquad\hbox{and}\qquad
n \prod\limits_{j=1}^N 
\abs{\xi_j-\lambda}\leq 2c_1 +
\prod\limits_{j=1}^N 
\abs{\xi_j-\lambda}\,\leq 2c_1 + 3^N.
$$
Then (\ref{EQSIFERT}) implies that 
$$
\prod\limits_{j=1}^N \abs{\xi_j-\lambda} \|R(\lambda,T)\| \leq
\frac12\, \prod\limits_{j=1}^N \abs{\xi_j-\lambda} \|R(\lambda,T)\| \leq
+ c_0\Bigl(n \prod\limits_{j=1}^N 
\abs{\xi_j-\lambda} + c\Bigr),
$$
and hence
$$
\prod\limits_{j=1}^N \abs{\xi_j-\lambda} \|R(\lambda,T)\| \leq 2 c_0\Bigl(n \prod\limits_{j=1}^N 
\abs{\xi_j-\lambda} + c\Bigr)\leq
2c_0\bigl(2c_1+3^N+c\bigr).
$$
This proves that $T$ is Ritt$_E$.
\end{proof}

\smallskip
\section{Polygonal functional calculus and $R$-Ritt$_E$ operators}
	
In this section we define polygonal functional calculus
and $H^\infty(E_s)$ functional calculus, 
as well as $R$-Ritt$_E$ operators. Our main results, Theorem \ref{thPolyHI}
and Corollary \ref{Main1}, assert that for any  $R$-Ritt$_E$ operator,
polynomial boundedness implies (and hence is equivalent to) either
bounded polygonal functional calculus or a bounded
$H^\infty(E_s)$ functional calculus. The proofs rely on a 
decomposition principle that 
we establish in Subsection \ref{Unity}.

\smallskip
\subsection{Functional calculi}
For any non empty open set $\Omega\subset\C$, we let 
$H^\infty(\Omega)$ denote the Banach algebra of all
bounded holomorphic functions $\phi\colon \Omega\to\C$,
equipped with
$$
\norm{\phi}_{\infty,\Omega} = 
\sup\bigl\{\vert\phi(\lambda)\vert\, :\, \lambda\in \Omega\bigr\}.
$$
Let $\P$ denote the algebra of all one variable complex polynomials.

Let $X$ be a Banach space and let $T\in B(X)$ such that
$\sigma(T)\subset\overline{\D}$. 
Following classical terminology, we say that $T$ is polynomially bounded if there exists
a constant $K\geq 1$ such that 
\begin{equation}\label{PB}
\norm{\phi(T)}\leq K \norm{\phi}_{\infty,\D},\qquad \phi\in\P.
\end{equation}
We recall the obvious fact that any polynomially bounded
operator is power bounded.

In the sequel, the name ``polygon" is reserved for open convex polygons, that is,
bounded finite intersections of open half-planes.

\begin{definition}\label{Polygonal}
We say that $T$ admits a bounded
polygonal functional calculus if there exist a polygon $\Delta\subset\D$ 
such that $\sigma(T)\subset\overline{\Delta}$, and a 
constant $K\geq 1$ such that 
\begin{equation}\label{Polygonal1}
\norm{\phi(T)}\leq K\norm{\phi}_{\infty,\Delta},\qquad \phi\in\P.
\end{equation}
\end{definition}

It is plain that $T$ is polynomially bounded if it admits a 
bounded polygonal functional calculus.
It essentially follows from \cite{LLM} that the converse is wrong (see Remark 
\ref{Wrong}).

\begin{remark}\label{PolygonToRitt}
Assume that $T$ satisfies Definition \ref{Polygonal}
for some polygon $\Delta\subset\D$. 
Let $E=\overline{\Delta}\cap \T$. 
This is a finite set, containing 
the peripheral spectrum $\sigma(T)\cap\T$.

Assume that $\sigma(T)\cap\T\not=\emptyset$.
Let 
$$
{\mathcal D}_{1,2}=\bigl\{ z\in\C\, :\, 1<\vert z\vert <2\bigr\}
$$
and define
$$
h(z,\lambda) = (z-\lambda)^{-1}\prod\limits_{\xi\in E}(\xi - z),\qquad
z\in {\mathcal D}_{1,2},\ \lambda\in\D.
$$
For any $\xi\in E$, 
the quotient $\frac{\xi- z}{z-\lambda}$ is bounded 
when $z,\lambda$ are close to $\xi$, $z\in {\mathcal D}_{1,2}$
and $\lambda\in\Delta$.

\begin{figure}[!h]
\includegraphics[scale=0.40]{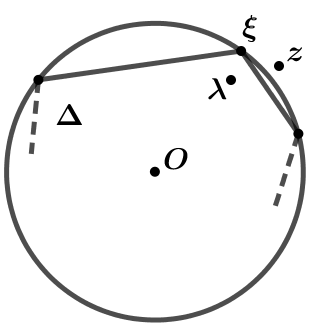}
 \caption{Positions of $\xi,z,\lambda$}
\label{fig:positions}
\end{figure}

This implies that
$C= \sup\bigl\{\norm{h(z,\,\cdotp)}_{\infty,\Delta}\, 
:\, z\in {\mathcal D}_{1,2}\bigr\}\, <\infty$.
Then we derive from (\ref{Polygonal1}) that  
$h(z,T) = R(z,T)\prod\limits_{\xi\in E}(\xi - z)$ satisfies
$\norm{h(z,T)}\leq KC$ for all $z\in {\mathcal D}_{1,2}$.
This implies that
$T$ is Ritt$_E$. Equivalently, see Remark \ref{RKradius}, (b),
$T$ satisfies a resolvent estimate
\begin{equation}\label{ResEst1}
\norm{R(z,T)} \lesssim \max\bigl\{
\vert z -\xi\vert^{-1}\, :\,\xi\in E\bigr\},
\qquad z\in\overline{\D}^c.
\end{equation}
Since $R(z,T)$ is bounded in the neighborhood
of any $\xi\in E$ not belonging to $\sigma(T)$, 
this implies an estimate (\ref{ResEst}) as well.

Of course if $\sigma(T)\cap\T=\emptyset$, then $R(z,T)$ is bounded on $\overline{\D}^c$.
\end{remark}

\bigskip
Let $E=\{\xi_1,\ldots,\xi_N\}\subset\T$ as in Section 2. 
Assume that $T$ is a Ritt$_E$ operator. In the sequel 
we will say that $T$ is a Ritt$_E$ operator of type 
$r\in(0,1)$ provided that it satisfies the conclusion
of Lemma \ref{LEstolz}.

For any $s\in(0,1)$, we let
$H_0^\infty(E_s)$ be the subspace of all $\phi\in H^\infty(E_s)$ for which
there exist positive real numbers $c,s_1,...,s_n >0$ such that : 
\begin{equation}\label{H0}
\abs{\phi(\lambda)} \leq c \prod\limits_{j=1}^N \abs{\xi_j-\lambda}^{s_i},
\end{equation}
for all $\lambda\in E_s$.

Assume that $T$ is a Ritt$_E$ operator of type $r\in(0,1)$ and let $s\in(r,1)$.
For any $\phi\in H_0^\infty(E_s)$, we set
$$
\phi(T) = \frac 1 {2\pi i}\int_{\partial E_u} \phi(\lambda) R(\lambda,T)\, d\lambda,
$$
where $u\in(r,s)$ and the boundary $\partial E_u$ of $E_u$ is oriented counterclockwise.
According to Remark \ref{Warning}, we have $\partial E_u\setminus E\subset E_s$ hence 
integration of $\phi(\lambda) R(\lambda,T)$ along $\partial E_u$ makes sense. Further
the integral is absolutely convergent, hence well-defined. Indeed this
follows from (\ref{H0}) and the estimate (ii) in Lemma \ref{LEstolz}. Furthermore 
by Cauchy's theorem,
the value of this integral does not depend on the choice of $u$. It is easy to check that
$H_0^\infty(E_s)$ is a subalgebra of $H^\infty(E_s)$ and 
that the mapping
$$
H_0^\infty(E_s)\longrightarrow B(X),\qquad 
\phi\mapsto \phi(T),
$$
is a homomorphism.

\begin{definition}\label{DEhinfty}
Let $T$ be a Ritt$_E$ operator of type $r\in(0,1)$ and let $s\in(r,1)$
We say that $T$ admits a bounded $H^\infty(E_s)$ functional 
calculus if there exists a constant $K\geq 1$
such that 
\begin{equation}\label{HFC}
\|\phi(T)\| \leq K \|\phi\|_{\infty,E_s},
\qquad\phi \in H_0^\infty(E_s).
\end{equation}
\end{definition}

The above definitions are natural extensions of the ones considered in
\cite{LM1} for Ritt operators. In this spirit, the following 
is an analogue of \cite[Proposition 2.5]{LM1}.

\begin{proposition}\label{PRhinftypoly}
Let $T$ be a  Ritt$_E$ operator of type $r \in (0,1)$ and let 
$s\in(r,1)$. Then $T$ has a 
bounded $H^\infty(E_s)$ functional calculus
if and only if there exists a constant $K \geq 1$ such that
$$
\|\phi(T)\| \leq K \|\phi\|_{\infty,E_s},\qquad \phi\in \P.
$$
\end{proposition}

\begin{proof}
The proof of the `if part' is identical to that of 
\cite[Proposition 2.5]{LM1}
so we skip it.

For the `only if' part, assume (\ref{HFC}). Consider (Lagrange)
polynomials $L_1,\ldots,L_N\in\P$ satisfying $L_i(\xi_j)=\delta_{i,j}$, 
for all $1\leq i,j\leq N$. Let $\psi\in \P$ and write $\psi=\psi_0 +\psi_1$, with
$$
\psi_0 = \sum\limits_{j=1}^N \psi(\xi_j)L_j
\qquad\hbox{and}\qquad
\psi_1 = \psi - \sum\limits_{j = 1}^N \psi(\xi_j)L_j.
$$
Then $\psi_1(\xi_j)=0$ for all $j=1,\ldots,N$, hence $\psi_1\in H^\infty_0(E_s)$.  Writing 
$\psi(T)=\psi_0(T) +\psi_1(T)$, and using $\psi_0(T)= \sum_j \psi(\xi_j)L_j(T)$, we infer
$$
\|\psi(T)\| \leq \|\psi_0(T)\| + \|\psi_1(T)\|\leq
\sum\limits_{j=1}^N\abs{\psi(\xi_j)}\|L_j(T)\|+ K\|\psi_1\|_{\infty,E_s}.
$$
Further, $\|\psi_1\|_{\infty,E_s} \leq \|\psi\|_{\infty,E_s} + \|\psi_0\|_{\infty,E_s}$ and
$$
\|\psi_0\|_{\infty,E_s}
\leq \sum\limits_{j=1}^N\abs{\psi(\xi_j)} \|L_j\|_{\infty,E_s}
\leq \Bigl(\sum\limits_{j=1}^N \|L_j\|_{\infty,E_s}\Bigr)
\|\psi\|_{\infty,E_s}.
$$
We derive that
$$
\|\psi(T)\| \leq\biggl(
K\Bigl(1+ \sum\limits_{j=1}^N \|L_j\|_{\infty,E_s}\Bigr) +\sum\limits_{j=1}^N\norm{L_j(T)}\biggr)
\norm{\psi}_{\infty, E_s},
$$
which proves the result.
\end{proof}

Definitions \ref{Polygonal} and \ref{DEhinfty} are connected by the following 
property. See also Remark \ref{Sharp} for more on this.

\begin{proposition}\label{H-P}
For any $T\in B(X)$, the following assertions are equivalent.
\begin{itemize}
\item [(i)] The operator $T$ admits a bounded polygonal functional 
calculus.
\item[(ii)] There exist a finite subset $E\subset \T$ and 
$s\in(0,1)$ such that 
$T$ is Ritt$_E$ and $T$ admits a bounded $H^\infty(E_s)$ functional calculus.
\end{itemize}
\end{proposition}

\begin{proof}
Assume (i), that is,
$T$ satisfies Definition \ref{Polygonal}
for some polygon $\Delta\subset\D$. We may assume that 
$E=\overline{\Delta}\cap \T$ is non empty. Then $T$
is Ritt$_E$ by Remark \ref{PolygonToRitt}.
Furthermore for $s\in(0,1)$ large enough, we have
$\Delta\subset E_s$. Hence by Proposition \ref{PRhinftypoly},
the estimate (\ref{Polygonal1}) implies that $T$
admits a bounded $H^\infty(E_s)$ functional calculus.

Assume (ii). It is plain that there exists 
a polygon $\Delta\subset\D$ such that
$$
E_s\subset \Delta\subset\D.
$$
Hence applying Proposition \ref{PRhinftypoly}, we obtain that $T$
satisfies an estimate (\ref{Polygonal1}).
\end{proof}

\smallskip
\subsection{Decomposition of unity}\label{Unity}
We fix $E=\{\xi_1,\ldots,\xi_N\}\subset\T$ as in Section 2. 
We let $H^\infty_{0,E}(\D)$ be the space of all $\phi\in H^\infty(\D)$ for which
there exist positive real numbers $c,s_1,...,s_n >0$ such that (\ref{H0}) 
holds true 
for all $\lambda\in \D$.

\begin{proposition}\label{DecPri}
There exist three sequences
$(\theta_i)_{i\geq 1}$, $(\phi_i)_{i\geq 1}$ and $(\psi_i)_{i\geq 1}$ of 
$H_{0,E}^\infty(\D)$ such that:
\begin{align*}
&(i)\ \sup\limits_{z \in \D} \sum\limits_{i = 1}^\infty\abs{\phi_i(z)} <\infty 
\qquad\hbox{and}\qquad
\sup\limits_{z \in \D} \sum\limits_{i = 1}^\infty\abs{\psi_i(z)} <\infty;\qquad\qquad\\
&(ii) \ \sup\limits_{i \geq 1}\, \sup\limits_{z \in \D} 
\abs{\theta_i(z)} <\infty;\qquad\qquad\\
&(iii))\ \forall\, r\in (0,1), 
\quad \sup\limits_{i\geq 1} \int_{\partial E_{r}} \frac {\abs{\theta_i(z)}}
{\prod\limits_{k=1}^N \vert\xi_k -z\vert} \,\abs{dz}< \infty;
\qquad\qquad\\
&(iv)\ \forall\, z \in \D,\quad 1 = \sum\limits_{i = 1}^\infty \theta_i(z)\phi_i(z)\psi_i(z).\qquad\qquad
\end{align*}
\end{proposition}

\begin{proof}
Let $H^\infty_{00}(\Sigma_{\frac{\pi}{2}})$ be the space 
of all $\Phi\in H^\infty(\Sigma_{\frac{\pi}{2}})$
for which there exist $c,s>0$ such that 
\begin{equation}\label{H00}
\vert\Phi(\lambda)\vert\leq c\vert \lambda\vert^s,\qquad\lambda\in
\Sigma_{\frac{\pi}{2}}.
\end{equation}
Then according to \cite[Proposition 6.3]{ArLM},
there exist three sequences
$(\Theta_i)_{i\geq 1}$, $(\Phi_i)_{i\geq 1}$ and $(\Psi_i)_{i\geq 1}$ of 
$H_{00}^\infty(\Sigma_{\frac{\pi}{2}})$ such that
\begin{equation}\label{SectCase1}
\sup\limits_{\lambda\in\Sigma_{\frac{\pi}{2}}} 
\sum\limits_{i = 1}^\infty\abs{\Phi_i(\lambda)} <\infty 
\qquad\hbox{and}\qquad
\sup\limits_{\lambda\in\Sigma_{\frac{\pi}{2}}} 
\sum\limits_{i = 1}^\infty\abs{\Psi_i(\lambda)} <\infty,
\end{equation}
\begin{equation}\label{SectCase2}
\sup\limits_{i \geq 1}\,
\sup\limits_{\lambda\in\Sigma_{\frac{\pi}{2}}} 
\abs{\Theta_i(\lambda)} <\infty,
\end{equation}
\begin{equation}\label{SectCase3}
\forall\, \nu\in\bigl(0,\tfrac{\pi}{2}\bigr), 
\quad \sup\limits_{i\geq 1} \int_{\partial \Sigma_\nu} 
\abs{\Theta_i(\lambda)}\,\frac{\vert d\lambda\vert}{\vert\lambda\vert}
\,<\infty,
\end{equation}
and
\begin{equation}\label{SectCase4}
\forall\,\lambda\in\Sigma_{\frac{\pi}{2}},
\quad 1 = \sum\limits_{i = 1}^\infty \Theta_i(\lambda)\Phi_i(\lambda)\Psi_i(\lambda).\qquad\qquad
\end{equation}
For any multi-index $\iota=(i_1,i_2,\ldots,i_N)\in\N^N$,
define $\theta_\iota,\phi_\iota, \psi_\iota
\colon\D\to\C$ by 
$$
\theta_\iota(z) = \Theta_{i_1}\bigl(1-\overline{\xi_1}z\bigr)
\Theta_{i_2}\bigl(1-\overline{\xi_2}z\bigr)
\cdots\Theta_{i_N}\bigl(1-\overline{\xi_N}z\bigr),
$$
and similarly
$$
\phi_\iota(z) = \prod\limits_{j=1}^N \Phi_{i_j}
\bigl(1-\overline{\xi_j}z\bigr)
\qquad \hbox{and}\qquad
\psi_\iota(z) = \prod\limits_{j=1}^N \Psi_{i_j}
\bigl(1-\overline{\xi_j}z\bigr).
$$
This is well-defined since for any $z\in\D$, 
we have $\overline{\xi_j}z\in\D$ hence
$1-\overline{\xi_j}z\in\Sigma_{\frac{\pi}{2}}$. Moreover
$$
\vert \theta_\iota(z)\vert
\leq c^N \prod\limits_{j=1}^N \vert 1-\overline{\xi_j}z\vert^s\,
=\, c^N \prod\limits_{j=1}^N \vert \xi_j- z\vert^s
$$
for any $z\in\D$, by (\ref{H00}). Hence all the functions $\theta_\iota$
belong to $H_{0,E}^\infty(\D)$. Likewise, 
all the functions $\phi_\iota$ and $\psi_\iota$ belong to $H_{0,E}^\infty(\D)$.
Since we can re-index the families 
$(\theta_\iota)_{\iota\in\N^N}$, $(\phi_\iota)_{\iota\in\N^N}$ and 
$(\psi_\iota)_{\iota\in\N^N}$ as sequences,
it suffices to show that they satisfy the properties (i)--(iv) of the statement.

For any $z\in\D,$
$$
\sum_{\iota=(i_1,\ldots,i_N)\in\N^N}
\vert 
\phi_\iota(z)\vert\, 
=\,\prod\limits_{j=1}^N\biggl(\sum_{i_j=1}^\infty\bigl\vert \Phi_{i_j}
\bigl(1-\overline{\xi_j}z\bigr)\bigr\vert \biggr)\,\leq\,
\biggl(\sup\limits_{\lambda\in\Sigma_{\frac{\pi}{2}}} 
\sum\limits_{i = 1}^\infty\abs{\Phi_i(\lambda)}\biggr)^N.
$$
Hence by (\ref{SectCase1}), the family 
$(\phi_\iota)_{\iota\in\N^N}$ satisfies (i). Likewise, 
the family $(\psi_\iota)_{\iota\in\N^N}$ satisfies (i), and
$(\theta_\iota)_{\iota\in\N^N}$ satisfies (ii), by (\ref{SectCase2}). 
Further (iv) holds true 
by (\ref{SectCase4}), since we have
$$
1=\prod\limits_{j=1}^N\Bigl(
\sum\limits_{i_j = 1}^\infty \Theta_{i_j}\bigl(1-\overline{\xi_j}z\bigr)
\Phi_{i_j}\bigl(1-\overline{\xi_j}z\bigr)\Psi_{i_j}\bigl(1-\overline{\xi_j}z\bigr)\Bigr)
=\,
\sum_{\iota=(i_1,\ldots,i_N)\in\N^N} 
\theta_\iota(z)\phi_\iota(z)\psi_\iota(z),
$$
for all $z\in \D$.

It therefore remains to check (iii). We fix $r\in(0,1)$. It 
follows from Definition \ref{ER} that there exist $C,\delta>0, 
\rho\in(0,1)$, as well as $2N$ angles $\nu_1,\nu_1',\ldots,\nu_N,\nu_{N}'$ in
$\bigl(0,\frac{\pi}{2}\bigr)$ so that the following holds true: 
\begin{itemize}
\item [(a)] For all $j=1,\ldots,N$, $\partial E_r\cap D(\xi_j,\delta)$ is the concatenation 
of the oriented segments $\bigl[\xi_j(1-\delta e^{i\nu_{j}}), \xi_j \bigr]$
and $\bigl[\xi_j, \xi_j(1-\delta e^{-i\nu_{j}'})\bigr]$;

\smallskip
\item [(b)] For all $z\in\partial E_r\setminus \cup_{j=1}^N D(\xi_j,\delta)$,
we have $\vert z\vert\leq\rho$.

\smallskip
\item [(c)] For all $1\leq j\not= k\leq N$, for all $z\in D(\xi_j,\delta)$, 
we have $\vert z-\xi_k\vert\geq C$.
\end{itemize}

Let $\iota=(i_1,\ldots,i_N)\in\N^N$.
Using (b), we have
$$
\int_{\partial E_{r}} \frac {\abs{\theta_{\iota}(z)}}
{\prod\limits_{k=1}^N \vert\xi_k -z\vert}\,\abs{dz}
\,\leq\frac{1}{(1-\rho)^N}\int_{\partial E_{r}} \abs{\theta_{\iota}(z)}
\abs{dz}\, + \,\sum_{j=1}^N
\int_{\partial E_{r}\cap D(\xi_j,\delta)} \frac {\abs{\theta_{\iota}(z)}}
{\prod\limits_{k=1}^N \vert\xi_k -z\vert} \,\abs{dz}.
$$
By (ii), the first term in the right hand-side is uniformly bounded. According 
to (a), it therefore suffices to show that for any $j=1,\ldots, N$, the integrals
\begin{equation}\label{ToBeUB}
\int_{[\xi_j(1-\delta e^{i\nu_{j}}), \xi_j]} \frac{\abs{\theta_{\iota}(z)}}
{\prod\limits_{k=1}^N \vert\xi_k -z\vert} \,\abs{dz}
\qquad\hbox{and}\qquad
\int_{[\xi_j, \xi_j(1-\delta e^{-i\nu_{j}'})]} \frac{\abs{\theta_{\iota}(z)}}
{\prod\limits_{k=1}^N \vert\xi_k -z\vert} \,\abs{dz}
\end{equation}
are uniformly bounded. 
Let $J_{\iota,j}$ be the first of these two integrals.
Let $K=\sup_{i}\norm{\Theta_i}_{\infty,\frac{\pi}{2}}$, given by 
(\ref{SectCase2}). Then by (c), we have
$$
J_{\iota,j} \leq\,\Bigl(\frac{K}{C}\Bigr)^{N-1}\,
\int_{[\xi_j(1-\delta e^{i\nu_{j}}), \xi_j]} 
\frac{\vert \Theta_{i_j} (1 - \overline{\xi_j}z)\vert }{\vert\xi_j -z\vert} 
\abs{dz}.
$$
Moreover the change of variable $\lambda= 1 - \overline{\xi_j}z$ leads to
$$
\int_{[\xi_j(1-\delta e^{i\nu_{j}}), \xi_j]} 
\frac{\vert \Theta_{i_j} (1 - \overline{\xi_j}z)\vert }{\vert\xi_j -z\vert} 
\,\abs{dz}\,\leq\,
\int_{\partial\Sigma_{\nu_j}} \vert\Theta_{i_j}(\lambda)\vert\,
\frac{\vert d\lambda\vert}{\vert \lambda\vert}.
$$
Property (\ref{SectCase3}) ensures that these integrals 
are uniformly bounded, hence the $J_{\iota,j}$ are uniformly bounded. 
Likewise, the second integrals in (\ref{ToBeUB}) are uniformly bounded, 
which concludes the proof.
\end{proof}

\smallskip
\subsection{From polynomial boundedness to bounded
polygonal functional calculus}\label{Reduction}

In the sequel we will use $R$-boundedness. We refer to \cite[Chapter 8]{HVVW}
for general information on this notion. We only recall basic notations and
the main definition. 
We let $(\epsilon_i)_{i\geq1}$ be a family of independent
Rademacher variables of some probability space $(\mathcal M,\mathbb{P})$. 
For any finite family $x_1,\ldots,x_n$ in $X$, we set
$$
\left\|\sum\limits_{i=1}^n \epsilon_i \otimes x_i\right\|_{Rad(X)} =
\biggl(\int_{\mathcal M} \Bignorm{\sum\limits_{i=1}^n \epsilon_i(t)x_i}_X^2
d\mathbb{P}(t)\biggr)^\frac12.
$$
Then we say that a subset $F\subset B(X)$ is $R$-bounded if there exists a constant $K 
\geq 0$ such that, for any $n\geq 1$, any $T_1,\ldots,T_n$ in $F$
and any $x_1,\ldots,x_n$ in $X$, 
$$
\left\|\sum\limits_{i=1}^n \epsilon_i \otimes T_i(x_i)\right\|_{Rad(X)}
\leq K\left\|\sum\limits_{i=1}^n \epsilon_i \otimes x_i\right\|_{Rad(X)}.
$$
In this case, we let ${\mathcal R}(F)$ denote the smallest possible $K 
\geq 0$ verifying this property.

\bigskip
In the sequel we fix again 
$E=\{\xi_1,\ldots,\xi_N\}\subset\T$ as in Section 2. 

\begin{definition} 
We say that an operator $T\in B(X)$ is $R$-Ritt$_E$ if 
$\sigma(T) \subset \overline{\D}$ and the set
$$
\biggl\{R(z,T) \prod\limits_{j=1}^N(\xi_j-z)\, :\, z\in\C,\ 1<\vert z\vert <2\biggr\}
$$
is $R$-bounded.
\end{definition}

\begin{remark}\label{RRE}
Let $T\in B(X)$ such that $\sigma(T)\subset \overline{\D}$.

\smallskip
(a)
For each $j=1,\ldots,N$, let ${\mathcal V}_j$ be an open neighborhood of $\xi_j$
such that $\overline{{\mathcal V}_j}\cap \overline{{\mathcal V}_k}=\emptyset$
if $j\not= k$.
If $\sigma(T)\cap\T\subset E$, then the set 
$$
W=\bigl\{z\in\C\, :\, 1\leq \vert z\vert \leq 2
\bigr\}\setminus \bigcup_{j=1}^N {\mathcal V}_j
$$
is  compact and $W\cap\sigma(T)=\emptyset$.
Hence 
$$
\biggl\{R(z,T) \prod\limits_{j=1}^N(\xi_j-z)\, :\, z\in
W\biggr\}
$$
is $R$-bounded, by \cite[Proposition 2.6]{W}. 
Thus $T$ is $R$-Ritt$_E$ if and only if for all $j=1,\dots,N$, the set 
$\{(\xi_j-z)R(z,T)\, :\, 1<\vert z\vert <2,\ z\in{\mathcal V}_j\}$ is $R$-bounded.
Furthermore by \cite[Proposition 2.8]{W}, the sets 
$\{(\xi_j-z)R(z,T)\, :\, \vert z\vert\geq 2\}$ are $R$-bounded.
(Here we use the disjointness of the $\overline{{\mathcal V}_j}$.)
Hence $T$ is $R$-Ritt$_E$ if and only if the sets
$$
\bigl\{(\xi_j-z)R(z,T)\, :\, z\in{\mathcal V}_j\cap\overline{\D}^c\bigr\},
\qquad j=1,\ldots,N,
$$
are $R$-bounded.

\smallskip (b) Arguing as in Lemma \ref{LEstolz}, one also obtains
that $T$ is $R$-Ritt$_E$ if and only
there exists an $E$-large enough 
$r\in (0,1)$ such that
$$
\sigma(T) \subset \overline{E_r}
$$
and 
for all $s\in(r,1)$, the set
$$
\biggl\{R(z,T) \prod\limits_{j=1}^N(\xi_j-z)\, :\, z\in D(0,2)\setminus \overline{E_s}\biggr\}
$$
is $R$-bounded.
\end{remark}

\begin{theorem}\label{thPolyHI}
Let $T \in B(X)$ be an $R$-Ritt$_E$ operator. If $T$ is polynomially bounded, then
$T$ admits a bounded $H^\infty(E_s)$ functional calculus for some $s\in (0,1)$.
\end{theorem}

Combining this theorem with 
Proposition \ref{H-P}, we immediately deduce the following result.

\begin{corollary}\label{Main1}
Let $T \in B(X)$ be an $R$-Ritt$_E$ operator. 
If $T$ is polynomially bounded, then
$T$ admits a bounded polygonal functional calculus.
\end{corollary}

\begin{proof}[Proof of Theorem \ref{thPolyHI}] 
Let $r\in(0,1)$ such that $T$ satisfies Remark \ref{RRE}, (b), and fix some
$s \in (r,1)$. 
Consider the three sequences $(\theta_i)_{i\geq1}$, $(\phi_i)_{i\geq1}$ 
and $(\psi_i)_{i\geq1}$ 
provided by Proposition \ref{DecPri}.
Let $h\in H^\infty_0(E_s)$. Applying part (iv) of Proposition \ref{DecPri}, we have
$$
h(z) = \sum\limits_{i=1}^{\infty} h(z)\theta_i(z)\phi_i(z)\psi_i(z),
\qquad z\in\D.
$$
Fix some $u\in(r,s)$. According to parts (i) and (ii) of Proposition \ref{DecPri}, 
and the fact that $h\in H^\infty_0(E_s)$, we have
$$
\sum\limits_{i=1}^{\infty} \int_{\partial E_u}
\vert h(z)\theta_i(z)\phi_i(z)\psi_i(z)\vert \norm{R(z,T)}\,\vert dz\vert\,<\infty.
$$
Hence 
$$
h(T) =\,\sum\limits_{i=1}^{\infty} (h\theta_i\phi_i\psi_i)(T)
\,= \,\sum\limits_{i=1}^{\infty} h(T)\theta_i(T)\phi_i(T)\psi_i(T).
$$
For any $i\geq 1$, we may write
\begin{align*}
h(T)\theta_i(T) &= \frac 1 {2\pi i} \int_{\partial E_u} h(z) 
\theta_i(z) R(z,T) \, d{z}\\
&= \frac 1 {2\pi i} \int_{\partial E_u} \Biggl(\frac{h(z) 
\theta_i(z)}{\prod\limits_{j=1}^N(\xi_j-z)}\Biggr)\, 
\prod\limits_{j=1}^N (\xi_j-z) R(z,T) \, dz\,.
\end{align*}
Further we have an estimate
$$
\int_{\partial E_u}\Biggl\vert
\frac{ h(z) \theta_i(z)}{\prod\limits_{j=1}^N(\xi_j-z)}
\Biggr\vert\, \vert dz\vert\,\lesssim\norm{h}_{\infty, E_s},
$$
by part (iii) of Proposition \ref{DecPri}.
Applying \cite[Theorem 8.5.2]{HVVW}, we deduce that 
the set of all $h(T)\theta_i(T)$ is $R$-bounded, 
with
\begin{equation}\label{R}
\mathcal{R}\bigl(
\bigl\{h(T)\theta_i(T)\, :\, 
i\geq 1\bigr\}\bigr) \lesssim \|h\|_{\infty,E_s}.
\end{equation}
Fix $x \in X$ and $y \in X^{*}$. Set 
$h_n=\sum_{i=1}^n h\theta_i\phi_i\psi_i$ for any $n\geq 1$.
Then we have
\begin{align*} 
\langle{h_n(T)x,y}\rangle 
& =\sum\limits_{i=1}^n  \langle{h(T)\theta_i(T)\phi_i(T)x,\psi_i(T)^*y}\rangle\\
& \int_{\mathcal M} \Bigl\langle
\sum\limits_{i=1}^n\epsilon_i(t) h(T)\theta_i(T)\phi_i(T)x,
\sum\limits_{i=1}^n\epsilon_i(t)
\psi_i(T)^*y\Bigr\rangle
\, d\mathbb{P}(t).
\end{align*}
Applying the  
Cauchy-Schwarz inequality,
we deduce the inequality
$$
\abs{\langle{h_n(T)x,y}\rangle}\,\leq\,
\left\|\sum\limits_{i=1}^n \epsilon_i\otimes  h(T)\theta_i(T)\phi_i(T)x\right\|_{Rad(X)}
\left\|\sum\limits_{i=1}^n \epsilon_i\otimes \psi_i(T)^{*}y\right\|_{Rad(X^{*})}.
$$
Now applying (\ref{R}), this implies
$$
\abs{\langle{h_n(T)x,y}\rangle} \lesssim
\|h\|_{\infty,E_s} \left\|\sum\limits_{i=1}^n \epsilon_i\otimes \phi_i(T)x\right\|_{Rad(X)}
\left\|\sum\limits_{i=1}^n \epsilon_i\otimes \psi_i(T)^{*}y\right\|_{Rad(X^{*})}.
$$

By assumption, $T$ satisfies (\ref{PB}) for some
$K\geq 1$. Arguing as in \cite[Proposition 2.5]{LM1}, this implies that 
$$
\norm{\phi(T)}\leq K \norm{\phi}_{\infty,\D},\qquad \phi\in H^\infty_0(E_s).
$$
Hence applying part (i) of Proposition \ref{DecPri}, we have
an estimate
$$
\Bignorm{\sum_{i=1}^n \epsilon_i(t)\phi_i(T)}\lesssim 1,\qquad
t\in{\mathcal M},\, n\geq 1.
$$
This readily implies 
$$
\left\|\sum\limits_{i=1}^n \epsilon_i\otimes \phi_i (T)x\right\|_{Rad(X)} 
\lesssim \norm{x}.
$$
Similarly we have
$$
\left\|\sum\limits_{i=1}^n \epsilon_i\otimes\psi_i(T)^{*}y\right\|_{Rad(X^{*})} \lesssim \|y\|.
$$
Thus we have an estimate
$$
\vert \langle h_n(T)x,y\rangle\vert\lesssim
\|h\|_{\infty,E_s}\|x\|\|y\|.
$$
Since $h_n(T)\to h(T)$ when $n\to\infty$, we deduce the estimate
$$
\|h(T)\| \lesssim \|h\|_{\infty,E_s},
$$
which proves the bounded $H^\infty(E_s)$ functional calculus.
\end{proof}

Theorem \ref{thPolyHI} above was stated and proved for Ritt
operators in \cite[Proposition 7.7]{LM1}. The multi-point version
considered here required a different proof.

\begin{remark}\label{Hilbert}
Recall that if $X=H$ is a Hilbert space,  any bounded
subset of $B(H)$ is automatically $R$-bounded. Hence any
Ritt$_E$ operator $T\in B(H)$ is automatically $R$-Ritt$_E$.
It therefore follows from Corollary \ref{Main1} that
if $T\in B(H)$ is a polynomially bounded Ritt$_E$ operator,
then it admits a bounded polygonal functional calculus. 
This Hilbertian case, which was the motivation to undertake this work,
is due to de Laubenfels \cite[Theorem 4.4, $(a)\Rightarrow (b)$]{dL}.
It is also implicit in the Franks-McIntosh result  \cite[Theorem 5.5]{FMI}. 
Note that Proposition \ref{DecPri} relies on \cite[Proposition 6.3]{ArLM}, which
is itself a consequence of a construction devised in \cite{FMI}.
Thus in spirit, the proof of Corollary \ref{Main1} is closer  to \cite{FMI}
than to \cite{dL}.
\end{remark}

\begin{remark}\label{Delta}
Let $T$ be a Ritt$_E$ operator and recall the operators $A_j$ 
defined by (\ref{Aj}).
For $s\in(0,1)$ close enough to $1$, we have an inclusion  
$E_s\subset\Sigma(\xi_j,\arcsin(s))$ for 
all $j=1,\ldots,N$. It easily follows that if $T$ 
admits a bounded $\HI(E_s)$ functional calculus, then each
$A_j$
admits a bounded $\HI(\Sigma_{\arcsin(s)})$ functional calculus. 

Assume that $X$ has the triangular contraction property, in the sense 
of \cite[Section 7.5.b]{HVVW}. If $T$ admits a bounded $\HI(E_s)$ functional calculus
for some $s\in(0,1)$, it follows from above and from \cite[Theorem 5.3]{KW} that
the $A_j$ are $R$-sectorial of $R$-type $<\frac{\pi}{2}$. Using (\ref{Transfer3}) and 
applying Remark \ref{RRE}, we deduce that $T$ is $R$-Ritt$_E$.

Combining the above paragraph and Theorem \ref{thPolyHI}, we deduce that 
if  $X$ has the triangular contraction property, then  $T$ admits a bounded $\HI(E_s)$ functional calculus
for some $s\in(0,1)$ if and only if $T$ is both polynomially bounded and $R$-Ritt$_E$.
\end{remark}


\section{Ritt$_E$ contractively regular operators on $L^p$-spaces}
Our aim is to show that for any $1<p<\infty$, contractively regular operators on $L^p$ admit a bounded polygonal functional calculus if they 
are Ritt$_E$ for some $E$. This will be achieved in Subsection \ref{LP}.
In the previous Subsection \ref{Inter},
we establish a result of independent interest (valid on all Banach spaces)
connecting $H^\infty(E_s)$ functional calculus to the classical
$H^\infty$ functional calculus of sectorial operators.

\smallskip
\subsection{Intersection of sectorial functional calculi}\label{Inter}

In Subsection \ref{Aux} we recalled the definition of a sectorial operator
$A$ of type $\omega\in(0,\pi)$ in the case when $A$ is bounded (we will not
need unbounded sectorial operators in this paper). We will now use 
the notion of bounded $H^\infty(\Sigma_\theta)$ functional calculus for
any $A$ as above and $\theta\in(\omega,\pi)$. We refer to \cite[Section 2]{W1}, \cite[Chapter 5]{Haa} or 
\cite[Section 10.2.b]{HVVW} for the relevant definitions. These
three references provide a comprehensive information 
on the $H^\infty$ functional calculus of sectorial operators.

The following simple fact is given by \cite[Proposition 10.2.21]{HVVW}.

\begin{lemma}\label{BouInv}
Let $A\in B(X)$ and assume that 
$\sigma(A)\subset \Sigma_\omega$ for some $\omega\in(0,\pi)$.
Then $A$ is sectorial of type $\omega$ and for any $\theta\in(\omega,\pi)$, $A$ admits 
a bounded $H^\infty(\Sigma_\theta)$ functional calculus.
\end{lemma}

For the next statement, we note that if $A\in B(X)$ is sectorial of type $\omega\in(0,\pi)$,
then for 
all $\rho \in (0,1)$, we have
$$
\sigma((1-\rho)I_X + \rho A) \subset \Sigma_{\omega}.
$$
Hence for all $g \in \HI(\Sigma_\omega)$, the operator
$g\bigl((1-\rho)I_X + \rho A\bigr)$ is well-defined by
the Dunford-Riesz functional calculus.

\begin{lemma}\label{rho}
Let $A \in B(X)$ be a sectorial operator and assume
that $A$ admits a bounded $\HI(\Sigma_\theta)$ functional 
calculus for some $\theta\in(0,\pi)$. 
Then there exists a constant $K>0$ such that for all 
$g \in \HI(\Sigma_\theta)$ and for all $\rho \in (0,1)$,
\begin{equation}\label{A-rho}
\bignorm{g\bigl((1-\rho)I_X + \rho A\bigr)} \leq K \|g\|_{\infty,\Sigma_\theta}.
\end{equation}
\end{lemma}

\begin{proof}
Let ${\rm Rat}_\theta\subset H^\infty(\Sigma_\theta)$ be the algebra
of all rational functions with nonpositive degree and poles off
$\overline{\Sigma_\theta}$. Since $A$ admits a bounded $\HI(\Sigma_\theta)$ functional calculus,
there exists a constant $K>0$ such that
$$
\norm{\phi(A)}\leq K\norm{\phi}_{\infty,\Sigma_\theta},\qquad \phi\in {\rm Rat}_\theta.
$$
Let $\rho \in (0,1)$. Set $\varphi_\rho(z) = (1-\rho) +\rho z$ and observe that 
$\varphi_\rho(\Sigma_\theta)\subset\Sigma_\theta$.
We set $A_\rho= \varphi_\rho(A)$. For any $\psi\in{\rm Rat}_\theta$, we have 
$\psi(A_\rho)=(\psi\circ\varphi_\rho)(A)$, hence
$$
\norm{\psi(A_\rho)}\leq K \norm{\psi\circ\varphi_\rho}_{\infty,\Sigma_\theta}\leq K 
\norm{\psi}_{\infty,\Sigma_\theta}.
$$
Thus (\ref{A-rho}) is satisfied by any element of ${\rm Rat}_\theta$.
By an entirely classical argument (see \cite{LM0} or \cite[Section 5.3.4]{Haa}) , we deduce (\ref{A-rho}) for all $g \in \HI(\Sigma_\theta)$.
\end{proof}

In the sequel, we fix
a finite set $E=\{\xi_1,\ldots,\xi_N\}\subset\T$ as in Section 2.

\begin{theorem}\label{THhinftyLp}
Let $T\in B(X)$ be a Ritt$_E$ operator. For any $j=1,\ldots,N$, let
$A_j = I_X - \overline{\xi_j}T$ and assume that
there exists $\theta_j \in \bigl(0,\frac \pi 2\bigr)$ 
such that $A_j$ admits a bounded $\HI(\Sigma_{\theta_j})$ functional calculus. Then,
\begin{itemize}
\item [(i)] There exists $s\in (0,1)$ such that $T$ 
admits a bounded $\HI(E_s)$ functional calculus;
\item [(ii)] $T$ admits a bounded polygonal functional
calculus.
\end{itemize}
\end{theorem}

\begin{proof}
We are going to build a (convex, open) polygon $\Delta\subset\D$ with the following three properties:
\begin{itemize}
\item [$(\bullet)$] There exists a finite set $E'\subset\D$ such that the set of vertices
of $\Delta$ is equal to $E\cup E'$.
\item [$(\bullet\bullet)$] $\sigma(T)\subset \overline{\Delta}$.
\item [$(\bullet\bullet\bullet)$] There exists a constant $K\geq 1$ such that 
$\norm{\phi(T)}\leq K\norm{\phi}_{\infty,\Delta}$ for all $\phi\in\P$.
\end{itemize}
This will obviously prove part (ii) of the statement. This will also 
prove part (i), since for any polygon $\Delta$ satisfying $(\bullet)$, there
exists $s\in (0,1)$ such that $\Delta\subset E_s$.

Recall $\Sigma(\xi,\omega)$ defined by (\ref{Sigma-xi}) for any $\xi\in\C^*$ and
any $\omega\in\bigl(0,\frac{\pi}{2}\bigr)$. We introduce
$$
\partial\Sigma(\xi,\omega)_+ = \bigl\{\xi(1-te^{-i\omega})\, :\, t>0\bigr\}
$$
and
$$
\partial\Sigma(\xi,\omega)_- = \bigl\{\xi(1-te^{i\omega})\, :\, t>0\bigr\},
$$
see Figure \ref{fig:polygonal sector}. 

\begin{figure}[!h]
    \includegraphics[scale=0.3]{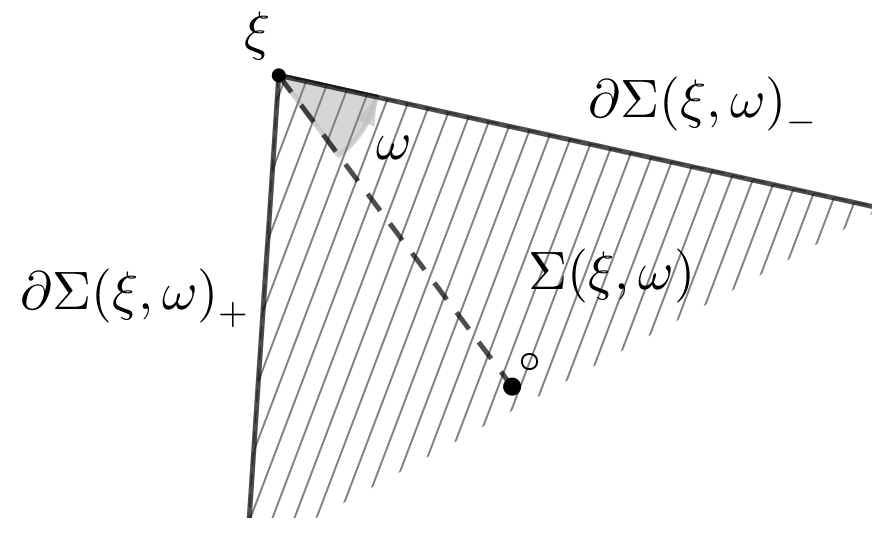}
    \caption{The set $\Sigma(\xi,\omega)$ and its boundary}
    \label{fig:polygonal sector}
\end{figure}

Assume that $N \geq 2$ (it is easy to adapt the proof to the case $N=1$).
For convenience we assume that the sequence $(\xi_1,\xi_2,\ldots,\xi_N)$
is oriented counterclockwise on $\T$ and we set $\xi_{N+1}=\xi_1$. 
We may choose $\theta\in\bigl(0,\frac{\pi}{2}\bigr)$ close enough to
$\frac{\pi}{2}$ so that : for all $j=1,\ldots,N$, the
operator $A_j$ have a bounded 
$\HI(\Sigma_{\theta})$ functional calculus; for all  $1\leq j\not= j'\leq N$,
$\xi_{j'} \in \Sigma(\xi_j,\theta)$; for all $j=1,\ldots,N$, the half-lines
$\partial \Sigma(\xi_j,\theta)_+$ and $\partial \Sigma(\xi_{j+1},\theta)_-$
do not meet in $\overline{\D}$.

Let us momentarily focus on the couple $(\xi_1,\xi_2)$, see Figure \ref{fig:polygon}. 
Let $\Gamma_{1,2}$ be the closed arc of $\T$ joining the points where 
$\Sigma(\xi_1,\theta)_+$ and $\Sigma(\xi_2,\theta)_-$ meet $\T$.
Then ${\rm dist}(\Gamma_{1,2},\sigma(T))>0$ hence by compactness,
we can find $r \in (0,1)$, $\theta' \in \bigl(0,\frac \pi 2\bigr)$ and 
points $z_1,...,z_p$ on $r\Gamma_{1,2}$, ordered counterclockwise, such that : 
\begin{itemize}
\item[-] For all $i=1,\ldots,p$, $\sigma(T) \subset \Sigma(z_i,\theta')$;
\item[-] The half-lines $\partial \Sigma(\xi_1,\theta)_+$ and $\partial \Sigma(z_1,\theta')_-$ meet in $\D\setminus\{0\}$;
\item[-] The half-lines $\partial \Sigma(z_p,\theta')_+$ and $\partial \Sigma(\xi_2,\theta)_-$ meet in $\D\setminus\{0\}$;
\item[-] For all $i=1,\ldots,p-1$, 
$\partial \Sigma(z_i,\theta')_+$ and $\partial \Sigma(z_{i+1},\theta')_-$ 
meet in $\D\setminus\{0\}$.
\end{itemize}

\begin{figure}[!h]
    \includegraphics[scale=0.26]{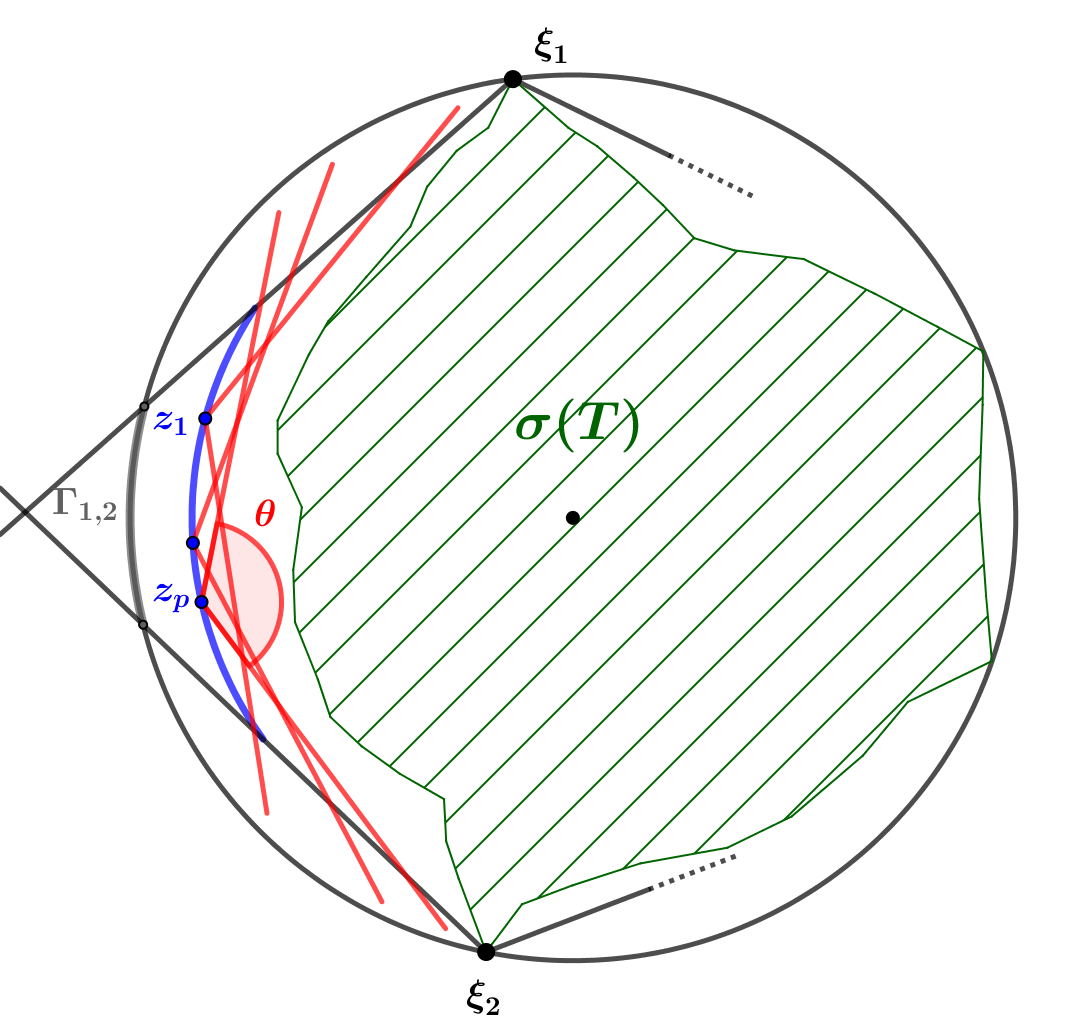}
    \caption{Construction of the polygon}
    \label{fig:polygon}
\end{figure}

Then we apply the same process to 
the couples $(\xi_2,\xi_3), ...,(\xi_N,\xi_1)$. Putting together
the points $\xi_j$ and the intermediate points 
$z_i$, we therefore obtain a finite
sequence $(\zeta_1,\zeta_2,\ldots,\zeta_m)$ of distinct elements of $\overline{\D}\setminus\{0\}$,
ordered counterclockwise,
as well as angles $\mu_1,...,\mu_m$ in $\bigl(0,\frac \pi 2\bigr)$, verifying the following properties:
\begin{itemize}
\item[(a)] We have 
$$
\bigl\{\zeta_1,\zeta_2,\ldots,\zeta_m\bigr\}\cap\T = E;
$$
\item[(b)] For any $i=1,\ldots,m$ : 
\begin{itemize}
\item[(b1)] If there exists $j \in \{1,...,N\}$ such that $\zeta_i = \xi_j$, then $\mu_i = \theta$;
\item[(b2)] If $\zeta_i \notin E$, then $\sigma(T) \subset \Sigma(\zeta_i,\mu_i)$
\end{itemize}
\item[(c)] Setting $\zeta_{m+1} =\zeta_1$, the half-lines  
$\partial\Sigma(\zeta_i,\mu_i)_+$ and $\partial\Sigma(\zeta_{i+1},\mu_{i+1})_-$ meet exactly at one point 
$c_i \in \D\setminus\{0\}$, for all $i=1,\ldots,m$.
\end{itemize}
Finally we set
$$
d_i = \frac 1 2 \left(c_i +  c_i\abs{c_i}^{-1}\right),
\qquad i=1,\ldots,m.
$$
We let $\Delta_0$ be the open polygon with vertices 
$\{\zeta_1,c_1,\zeta_2,c_2,...,\zeta_m,c_m\}$. We may assume that
it is convex.
Likewise, we let
$\Delta$ be the open polygon with vertices 
$\{\zeta_1,d_1,\zeta_2,d_2,...,\zeta_m,d_m\}$.
By construction we have
$$
\Delta_0 = \bigcap\limits_{i = 1}^m\Sigma(\zeta_i,\mu_i)
$$
and
$$
\Delta_0 \subset\Delta.
$$
According to (c), all the $d_i$ belong to $\D$. Hence it
follows from (a) that
the polygon $\Delta$ satisfies $(\bullet)$. 
It further satisfies $(\bullet\bullet)$, by (b).

Let us now show that the polygon $\Delta$ satisfies $(\bullet\bullet\bullet)$.
We set $d_0=d_{m}$ for convenience.
For all $i=1,\ldots,m$, we let $\gamma_i$ be the path obtained as the concatenation
of the oriented segments $[d_{i-1},\zeta_i]$ and $[\zeta_i,d_i]$.
Next we fix  $\epsilon >0$ small enough so 
that the sets $V_i = \Delta_0 \cap D(\zeta_i,\epsilon)$ 
are pairwise disjoint, see Figure \ref{fig:raised polygon}.

\begin{figure}[!h]
    \includegraphics[scale=0.3]{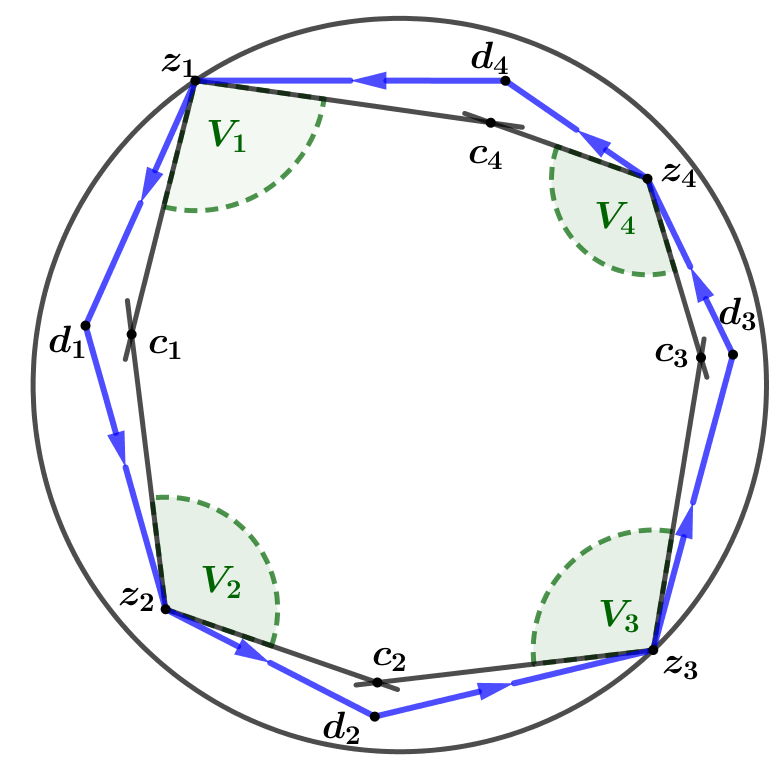}
    \caption{Polygons $\Delta_0$
    and $\Delta$ with $E=\{z_1,z_3\}$ and additional points $z_2, z_4$}
    \label{fig:raised polygon}
\end{figure}

Let $\phi$ be a polynomial.
For all $i=1,\ldots,m$, we define 
$\phi_i : \C\setminus \gamma_i \rightarrow \C$ by setting
$$
\phi_i(z) = \displaystyle{\frac 1 {2\pi i}
\int_{\gamma_i}\frac {\phi(\lambda)}{\lambda - z}
\,
d\lambda}.
$$
These functions are holomorphic. Moreover by Cauchy's theorem, we have 
\begin{equation}\label{EQcauchySum}
\forall\, z\in\Delta,\qquad \phi(z) = \sum\limits_{i=1}^m \phi_i(z).
\end{equation}

Since the distance from $\gamma_i$ to 
$\Sigma(\zeta_i,\mu_i)\setminus V_i$ is positive, 
we have an estimate 
\begin{equation}\label{First}
\abs{\phi_i(z)} \lesssim \|\phi\|_{\infty,\Delta},\qquad 
z \in \Sigma(\zeta_i,\mu_i)\setminus V_i,
\end{equation}
for all $i=1,\ldots,m$.
Next, observe that since the set $V_1$ is disjoint from each $V_i$, with $i\geq 2$,  we also
have estimates
$$
\abs{\phi_i(z)} \lesssim \|\phi\|_{\infty,\Delta},\qquad z\in V_1,\, 
i\geq 2.
$$
We have $V_1 \subset \Delta$ hence by (\ref{EQcauchySum}), 
$\phi_1(z) = \phi(z) - \sum\limits_{i=2}^m\phi_i(z)$ 
for all $z\in V_1$.
We deduce an estimate
$$
\abs{\phi_1(z)} \lesssim \|\phi\|_{\infty,\Delta},\qquad z\in V_1.
$$  
Combining with (\ref{First}) for $i=1$, we obtain that 
$\phi_1$ belongs to
$\HI(\Sigma(\zeta_1,\mu_1))$ and satisfies an estimate
$$
\|\phi_1\|_{\infty,\Sigma(\zeta_1,\mu_1)} \lesssim \|\phi\|_{\infty,\Delta}.
$$
A similar argument shows that for all $i=1,\ldots,m$,
\begin{equation}\label{Alli}
\phi_i \in \HI(\Sigma(\zeta_i,\mu_i))
\qquad\hbox{and}\qquad
\|\phi_i\|_{\infty,\Sigma(\zeta_i,\mu_i)} \lesssim \|\phi\|_{\infty,\Delta}.
\end{equation}
            
Now let $\rho \in (0,1)$. Since $(\bullet\bullet)$ holds true,
we have $\sigma(\rho T) \subset \Delta$.
We can therefore define operators
$\phi_i(\rho T)$ by the Dunford-Riesz functional calculus and we have 
\begin{equation}\label{Sum}
\phi(\rho T) = \sum\limits_{i = 1}^m \phi_i(\rho T),
\end{equation}
by (\ref{EQcauchySum}). (Here we use $\rho T$ instead of $T$ because
the $\phi_i(T)$ are a priori not defined.)

We shall now apply (b).
Let $i \in \{1,...,m\}$ and assume first 
that there exists 
$j \in \{1,...,N\}$ such that $\zeta_i = \xi_j$. 
Let $g\colon \Sigma_{\theta} \rightarrow \C$ be
defined by $g(z) = \phi_i \left(\xi_j(1-z)\right)$.
This function is well defined, holomorphic and bounded, 
by (\ref{Alli}). We further have
$$
\|g\|_{\infty,\Sigma_{\theta}} = \|\phi_i\|_{\infty,\Sigma(\xi_j,\theta)} = \|\phi_i\|_{\infty,\Sigma(\zeta_i,\mu_i)}.
$$
Since $A_j = I_X -\overline{\xi_j}T$, we have 
$$
g\left((1-\rho)I_X+\rho A_j\right) = \phi_i(\rho T).
$$
Applying Lemma \ref{rho}, we obtain an estimate
$$
\|\phi_i(\rho T)\| \lesssim \|\phi_i\|_{\infty,\Sigma(z_i,\mu_i)}.
$$
Appealing to (\ref{Alli}), we deduce an estimate
\begin{equation}\label{phi-i}
\|\phi_i(\rho T)\| \lesssim \|\phi\|_{\infty,\Delta}.
\end{equation}
Otherwise, $\zeta_i \notin E$ hence 
$\sigma(T) \subset \Sigma(z_i,\mu_i)$. Arguing as above and
using Lemma \ref{BouInv} instead of Lemma \ref{rho}, we obtain 
an estimate (\ref{phi-i}) as well.

We have $\|\phi(\rho T)\| \leq \sum\limits_{i=1}^m \|\phi_i(\rho T)\| $, 
by (\ref{Sum}), hence the estimates (\ref{phi-i}), for $i=1,\ldots,m$,
yield
$$
\|\phi(\rho T)\| \lesssim \|\phi\|_{\infty,\Delta}
$$
Since this estimate does not depend on
$\rho$ and $\phi(\rho T)\to\phi(T)$ when $\rho\to 1$, we obtain
property $(\bullet\bullet\bullet)$.
\end{proof}

\begin{remark}\label{Easy}
It follows from the first paragraph of
Remark \ref{Delta} and 
 Theorem
\ref{THhinftyLp} that 
 given a Ritt$_E$ operator
$T\in B(X)$, there exists $s\in (0,1)$ such that $T$ 
admits a bounded $\HI(E_s)$ functional calculus if and only if 
for each $j=1,\ldots,N$, 
there exists $\theta_j \in \bigl(0,\frac \pi 2\bigr)$ 
such that $A_j$ admits a bounded $\HI(\Sigma_{\theta_j})$ functional calculus. 
\end{remark}

\begin{remark}\label{Sharp}
We give here a complement to Proposition \ref{H-P}.
Let $T\in B(X)$ be a Ritt$_E$ operator. It follows from the previous remark 
and the proof of Theorem \ref{THhinftyLp} that if $T$ 
admits a bounded $\HI(E_s)$ functional calculus for some $s\in(0,1)$, then there
exists a polygon $\Delta\subset\D$ such that $T$ admits a bounded functional
calculus with respect to $\Delta$ and the set of vertices of $\Delta$ belonging to 
$\T$ coincides with $E$. This new condition on vertices is sharp.
\end{remark}

\begin{remark}\label{Wrong}
As a complement to Corollary \ref{Main1} and Remark \ref{Hilbert},
we mention that there 
exist  a Banach space $X$ and a Ritt$_E$ operator $T\in B(X)$ such that
$T$ is polynomially bounded but $T$ does not admit any bounded polygonal
functional calculus. To check this, assume (as we may do)
that $1\in E$. Recall the Stolz domains 
$$
B_\omega=\overset{\circ}{\rm Conv}\bigl(1, D(0,\sin(\omega))\bigr),
$$
for $\omega\in\bigl(0,\frac{\pi}{2}\bigr)$.
According to 
\cite[Theorem 3.2]{LLM}, there exists a Ritt operator $T$ on some $X$
which is polynomially bounded although it does not admit any bounded
$H^\infty(B_\omega)$ functional calculus. The operator $T$
is Ritt$_E$. Assume that $T$ admits a bounded functional
calculus with respect to some polygon $\Delta\subset\D$. Then 
$1\in\Delta$ and the argument in Remark \ref{Easy} implies
that $I_X-T$ admits a bounded $H^\infty(\Sigma_\theta)$ functional
calculus for some $\theta\in\bigl(0,\frac{\pi}{2}\bigr)$. 
This implies, by \cite[Proposition 4.1]{LM1}, that $T$ admits a  
bounded
$H^\infty(B_\omega)$ functional calculus for some 
$\omega\in\bigl(0,\frac{\pi}{2}\bigr)$, whence a contradiction.

\end{remark}

\smallskip
\subsection{Ritt$_E$ contractively regular 
operators on $L^p$ spaces}\label{LP}
We consider a measure space $(S,\mu)$, we let 
$1<p<\infty$ and we consider operators acting on 
$L^p(S)$. 

We recall that a bounded operator $T\colon L^p(S)\to L^p(S)$ is called 
regular if there exists a constant $C\geq 0$ such that for all
$n\geq 1$ and all $x_1,\ldots,x_n$ in $L^p(S)$, we have 
$$
\Bignorm{\sup\limits_{1\leq i \leq n}\abs{T(x_i)}}_p 
\leq C \Bignorm{\sup\limits_{1 \leq i \leq n}{\abs{x_i}}}_p.
$$
In this case, the smallest $C\geq 0$ verifying this
property is called the regular norm of $T$ and is denoted
by $\norm{T}_{r}$. 
We say that $T$ is contractively regular if $\norm{T}_{r}\leq 1$.

If $T\colon L^p(S)\to L^p(S)$ is positive, then $T$ is regular and $\norm{T}_{r}
=\norm {T}$. Thus positive contractions are contractively regular.

It is plain that the set  
$B_{reg}(L^p(S))$ of all regular operators equipped
with $\norm{\,\cdotp}_{r}$ is a Banach algebra.

We refer e.g. to \cite[Section 1]{Pis2} for information on regular operators.

\begin{theorem}\label{MainLP}
Let $T\colon L^p(S)\to L^p(S)$ be a Ritt$_E$ contractively regular operator, 
with $1<p<\infty$. Then,
\begin{itemize}
\item [(i)] There exists $s \in (0,1)$ such that $T$ 
admits a bounded $\HI(E_s)$ functional calculus;
\item [(ii)] $T$ admits a bounded polygonal functional
calculus.
\end{itemize}
\end{theorem}
        
\begin{proof} 
Consider $A_j = I_X-\overline{\xi_j} T$ for all $j=1,\ldots,N$ and define
$T_{j,t} := e^{-tA_j}$ for all $t\geq 0$. These operators are
all contractively regular. Indeed,
$$
\|T_{j,t}\|_{r} 
= e^{-t}\|e^{t\overline{\xi_j} T}\|_r
\leq e^{-t} \sum\limits_{k=0}^\infty 
\frac{t^k \|\overline{\xi_j}T\|_{r}^k}{k!}
\leq e^{-t} e^{t\norm{\overline{\xi_j} T}_{r}}\leq 1.
$$ 
Then, according to \cite[proposition 2.2]{LMX}, 
there exists, for all $j=1,\ldots,N$, some 
$\theta_j\in\bigl(0,\frac \pi 2\bigr)$ such that $A_j$ admits a $\HI(\Sigma_{\theta_j})$ bounded functional calculus. 
We conclude by applying Theorem \ref{THhinftyLp}.
\end{proof}

\bigskip
\noindent
{\bf Acknowlegement.} The authors were supported by the ANR project {\it Noncommutative analysis on groups and quantum groups}
(No./ANR-19-CE40-0002). Further, the LmB receives support from
the EIPHI Graduate School (contract ANR-17-EURE-0002)

\vskip 1cm
        
\bibliographystyle{abbrv}	

\vskip 0.8cm

\end{document}